\documentclass{amsart}

\newtheorem{thm}{Theorem}[section]
\newtheorem{prop}[thm]{Proposition}
\newtheorem{cor}[thm]{Corollary}
\newtheorem{lem}[thm]{Lemma}
\newtheorem{rem}[thm]{Remark}

\numberwithin{equation}{section}

\def \vtk{V^{\otimes k}}

\def \bY{\bar{Y}_\sigma}
\def \D{\Delta_k}
\def \Z{\mathbb Z}
\def \C{\mathbb C}

\def \N{\mathbb N}
\def \wt{{\rm wt}}
\def \Res{{\rm Res}}
\def \End{{\rm End}}

\def \mod{{\rm mod}}
\def \o{\omega}
\def \l{\lambda}

\newcommand{\ZZ}[0]{{\Z}_+}
\newcommand{\Lx}[0]{x^{j + 1} \frac{\partial}{\partial x}}
\newcommand{\Lo}[0]{x \frac{\partial}{\partial
x}}

\begin{document}

\title[Even order cyclic-twisted tensor product VOSA-modules]{Permutation-twisted modules for even order cycles acting on tensor product vertex operator superalgebras}

\author{Katrina Barron}
\address{Department of Mathematics, University of Notre Dame, Notre Dame, IN 46556}
\email{kbarron@nd.edu}
\author{Nathan Vander Werf}
\address{Department of Mathematics, University of Notre Dame,
Notre Dame, IN 46556}
\email{nvanderw@nd.edu}


\date{January 22, 2014}


\begin{abstract}
We construct and classify $(1 \; 2 \; \cdots \; k)$-twisted $V^{\otimes k}$-modules for $k$ even and $V$ a vertex operator superalgebra.  In particular, we show that the category of weak $(1 \; 2 \; \cdots \; k)$-twisted $V^{\otimes k}$-modules for $k$ even is isomorphic to the category of weak parity-twisted $V$-modules.  This result  shows that in the case of a cyclic permutation of even order, the construction and classification of permutation-twisted modules for tensor product vertex operator superalgebras is fundamentally different than in the case of a cyclic permutation of odd order, as previously constructed and classified by the first author.  In particular, in the even order case it is the parity-twisted $V$-modules that play the significant role in place of the untwisted $V$-modules that play the significant role in the odd order case.  
\end{abstract}

\maketitle

\section{Introduction}

Let $V$ be a  vertex operator (super)algebra, and for a fixed positive integer $k$,  consider the tensor product vertex operator (super)algebra $\vtk$ (see \cite{FLM3}, \cite{FHL}).  Any 
element $g$ of the symmetric group $S_k$ acts on $\vtk$ as a vertex operator (super)algebra automorphism, and thus it is appropriate  to consider $g$-twisted $\vtk$-modules.  
In \cite{BDM}, the author along with Dong and Mason constructed and classified the $g$-twisted $V^{\otimes k}$-modules for $V$ a vertex operator algebra.    In \cite{B-superpermutation-odd}, the first author constructed and classified the $(1 \;2 \; \cdots \; k)$-twisted $V^{\otimes k}$-modules for $V$ a vertex operator superalgebra and $k$ odd.  In addition, in \cite{B-superpermutation-odd} and \cite{BV-fermion}, we showed that the construction and classification for the case of an even order permutation is  fundamentally different in the super case than that for odd order permutations, and conjectured that parity-twisted $V$-modules instead of untwisted $V$-modules were playing a central role.   

In the present paper, as our main result, we give an explicit construction and classification of $(1 \; 2\; \cdots \; k)$-twisted $V^{\otimes k}$-modules for $k$ even and $V$ any vertex operator superalgebra.  In particular, we show that for $k$ even, the category of weak $(1 \; 2\; \cdots \; k)$-twisted $V^{\otimes k}$-modules is isomorphic to the category of weak {\it parity-twisted} $V$-modules.   Here the parity map $\sigma$ on any $\mathbb{Z}_2$-graded vector space is the identity on the even subspace and $-1$ on the odd subspace.   

This result is in contrast to the results of \cite{BDM} for vertex operator algebras and the results of \cite{B-superpermutation-odd} for vertex operator superalgebras for when $k$ is odd;  in these cases it was shown that the category of weak $(1 \; 2\; \cdots \; k)$-twisted $V^{\otimes k}$-modules is isomorphic to the category of weak  {\it untwisted} $V$-modules.  Thus this class of examples we construct and classify in this paper (i.e., for the case when $k$ is even and $V$ is a vertex operator superalgebra) are of fundamental importance in understanding the role of the parity map and parity-twisted modules in the theory of vertex operator superalgebras.

The results of this paper give formulas for the graded dimensions of $(1 \; 2 \; \cdots \; k)$-twisted $V^{\otimes k}$-modules in terms of the graded dimensions of parity-twisted $V$-modules, and vice versa, as in Corollary \ref{graded-dimension-corollary}.

Next we give some background on the theory of twisted modules in general, followed by some further comments on implications of the results of this paper for supersymmetric theories and for lattice vertex operator superalgebras.  
Twisted vertex operators were discovered and used in \cite{LW}.   Twisted modules for vertex operator algebras arose in the work of I. Frenkel, J. Lepowsky and A. Meurman \cite{FLM1}, \cite{FLM2}, \cite{FLM3} for the case of a lattice vertex operator algebra and certain lifts of the lattice isometry $-1$, in the course of the construction of the moonshine module vertex operator algebra (see also \cite{Bo}). This structure came to be understood as an ``orbifold model" in the sense of conformal field theory and string theory.  Twisted modules are the mathematical counterpart of ``twisted sectors", which are the basic building blocks of orbifold models in conformal field theory and string theory (see \cite{DHVW1}, \cite{DHVW2}, \cite{DFMS}, \cite{DVVV}, \cite{DGM}, as well as \cite{KS}, \cite{FKS}, \cite{Ba1}, \cite{Ba2}, \cite{BHS}, \cite{dBHO}, \cite{HO}, \cite{GHHO}, \cite{Ba3} and \cite{HH}).  Orbifold theory plays an important role in conformal field theory and in supersymmetric generalizations, and is also a way of constructing a new vertex operator (super)algebra from a given one.  

Formal calculus arising {}from twisted vertex operators associated to an even lattice was systematically developed in \cite{Le1}, \cite{FLM2}, \cite{FLM3} and \cite{Le2}, and the twisted Jacobi identity was formulated and shown to hold for these operators (see also \cite{DL2}).  These results led to the introduction of the notion of $g$-twisted $V$-module \cite{FFR}, \cite{D}, for $V$ a vertex operator algebra and $g$ an automorphism of $V$.  This notion records the properties of twisted operators obtained in \cite{Le1}, \cite{FLM1}, \cite{FLM2}, \cite{FLM3} and \cite{Le2}, and provides an axiomatic definition of the notion of twisted sectors for conformal field theory.  In general, given a vertex operator algebra $V$ and an automorphism $g$ of $V$, it is an open problem as to how to construct a $g$-twisted $V$-module.

A theory of twisted operators for integral lattice vertex operator superalgebras and finite automorphisms that are lifts (in a certain way) of a lattice isometry were studied in \cite{DL2} and \cite{Xu}, and the general theory of twisted modules for vertex operator superalgebras was developed by Li in \cite{Li2}.  Certain specific examples of permutation-twisted sectors in superconformal field theory have been studied from a physical point of view in, for instance, \cite{FKS}, \cite{BHS}, \cite{Maio-Schellekens1}, \cite{Maio-Schellekens2}. 

Here we would like to point out implications for the current work applied to lattice vertex operator superalgebras.  If $V_K$ is a vertex operator superalgebra associated to a positive-definite integral lattice $K$, then $V^{\otimes k}_K$ is the vertex operator superalgebra $V_L$ associated to the lattice $L$ where $L$ is the orthogonal direct sum of $k$ copies of $K$.  In the case when the lattice $K$ is even, $V_K$ is a vertex operator algebra, and a $k$-cycle permutation of $V^{\otimes k}_K = V_L$ is a lift of a lattice isometry of $L$ of order $k$.  Thus one can use either the construction of \cite{BDM} or the construction of Lepowsky, \cite {Le1}, \cite{DL2}, to develop a theory of $(1 \; 2 \; \cdots \; k)$-twisted $V^{\otimes k}_K$-modules.  This overlap of constructions was studied by the first author, along with Huang and Lepowsky, in \cite{BHL}.  One of the interests in this overlap is that it holds potential for understanding the geometric underpinnings of twisted theory.  In particular, it can be used to study the relationship  between the space-time geometric setting of the orbifolding used to construct the twisted sectors following the work first initiated by Lepowsky versus the worldsheet geometric setting of the orbifolding used to construct the twisted sectors following \cite{BDM}.   This overlap of constructions of $(1 \; 2 \; \cdots \; k)$-twisted $V^{\otimes k}_K$-modules for lattice vertex operator algebras holds also for the setting of lattice vertex operator superalgebras if and only if $k$ is odd.  That is if $K$ is integral instead of even as a lattice, and the positive integer $k$ is odd, the $k$-cycle permutation isometry acting on $L$ by permuting the $k$ copies of $K$ lifts to the $k$-cycle automorphism of the vertex operator superalgebra $V^{\otimes k}_K = V_L$.    Then one can use either the construction of \cite{B-superpermutation-odd} or the construction of Lepowsky extended to the super setting as in \cite{DL2}, \cite{Xu} (see also \cite{BV-fermion}) to develop a theory of $(1 \; 2 \; \cdots \; k)$-twisted $V^{\otimes k}$-modules.  

However, for the case of a lattice vertex operator superalgebra, when $k$ is even, one {\it can not} use the theory of \cite{DL2}, \cite{Xu} to construct the $(1 \; 2 \; \cdots \; k)$-twisted $V^{\otimes k}_K$-modules.  This is because in order to carry out the program in \cite{DL2}, \cite{Xu}, one must double the order of the lattice isometry in the $k$ even case.    But then the lift of the lattice permutation isometry $(1 \; 2 \; \cdots \; k)$ to an automorphism of the vertex operator superalgebra $V_L$ results in an automorphism of order $2k$; it does not result under this lifting to the $k$-cycle permutation automorphism of order $k$ acting on $V_K^{\otimes k}$.  The details of why this is true are given in \cite{BV-fermion}, in particular, in Remark 4.1.   Thus for $k$ even, the construction in the present paper is the {\it only} known construction and classification of these $(1 \; 2 \; \cdots \; k)$-twisted modules for lattice vertex operator superalgebras.    In fact, the construction and classification we give here can be used to help shed light on the open problem of how to construct and classify twisted modules for lattice vertex operator superalgebras in general, as well as for other vertex operator superalgebras and a general automorphism.

Another important application of the results of this paper comes from considering vertex operator superalgebras which are also supersymmetric.  This is the setting 
of two-dimensional, holomorphic, superconformal field theory, where the vertex operator superalgebras which describe genus-zero particle interactions have additional supersymmetric structure.  A supersymmetric vertex operator superalgebra is a vertex operator superalgebra that, in addition to being a representation of the Virasoro algebra, is also a representation of the $N=n$ Neveu-Schwarz algebra (a Lie superalgebra extension of the Virasoro algebra), where $n$ is the degree of supersymmetry.   See, e.g., \cite{Barron-announce}--\cite{Barron-n2twisted}.   In the case, when $V$ is an $N = n$ supersymmetric vertex operator superalgebra, for $n = 1,2$, the parity-twisted $V$-modules are naturally a representation of the $N = n$ Ramond algebra.  The $N=n$ Ramond algebra is another extension of the Virasoro algebra to a Lie superalgebra related to the $N=n$ Neveu-Schwarz algebra.  In physics terms, the (untwisted) modules for the supersymmetric vertex operator superalgebras are called the ``Neveu-Schwarz sectors" and the parity-twisted modules are called the ``Ramond sectors".  Then the permutation-twisted modules (i.e., the permutation-twisted sectors) for supersymmetric vertex operator superalgebras are the basic building blocks for permutation orbifold superconformal field theory. The current work along with the work of the first author in \cite{B-superpermutation-odd} implies that all permutation-twisted sectors arising in a permutation orbifold superconformal field theory are built up as tensor products of Neveu-Schwarz sectors (coming from the odd cycles) and Ramond sectors (coming from the even cycles).

The case of $(1 \; 2)$-twisted  $(V \otimes V)$-modules is an especially important class of twisted theories for supersymmetric vertex operator superalgebras since it is often the case, cf. \cite{Barron-n2twisted}, that an N=2 supersymmetric vertex operator superalgebra has the form $V\otimes V$ for $V$ an N=1 supersymmetric vertex operator superalgebra, and the transposition $(1 \;2)$ then is a ``mirror map".  This is the setting for ``mirror-twisted sectors" which give rise to representations of the ``topological N=2 superconformal algebra" (also called the ``twisted N=2 superconformal algebras"), yet another super extension of the Virasoro algebra.  See, for example  \cite{Barron-varna} and \cite{Barron-n2twisted}.   The present work shows that there is an intimate connection between these mirror-twisted modules and parity-twisted modules, i.e., between mirror-twisted sectors and Ramond sectors.  In particular, for N=2 supersymmetric vertex operator superalgebras of the form $V \otimes V$ as studied in \cite{Barron-n2twisted}, and for the mirror map $\kappa$ realized as the permutation $(1 \; 2)$ as given in \cite{Barron-n2twisted}, our main result in this paper implies that the category of weak $\kappa$-twisted $(V\otimes V)$-modules is isomorphic to the category of weak parity-twisted $V$-modules.  In other words, for this mirror map $\kappa$, the category of mirror-twisted sectors for $V \otimes V$ is isomorphic  to the category of N=1 Ramond twisted sectors of $V$.  The details of this application of the results of this paper are given by the first author in \cite{B-varna2013}.

One of the motivations for trying to use parity-twisted modules to construct permutation-twisted modules for even order cycles comes from the examples studied by the authors previously in \cite{BV-fermion} of $(1 \; 2 \; \cdots \; k)$-twisted $V^{\otimes k}$-modules for $k$ even and $V$ the one free fermion vertex operator superalgebra following the work of Dong and Zhao in \cite{DZ2}.    This construction for the special case of free fermions studied in \cite{BV-fermion} is completely different than that which we develop in the present paper or as the first author developed in \cite{B-superpermutation-odd}.  But the shape of the classification in \cite{BV-fermion} and the graded dimensions calculated in this work, led us to conjecture that the permutation-twisted modules for even order permutations could be achieved in general using parity-twisted modules.  In this paper we prove this conjecture, by directly constructing a $(1 \; 2 \; \cdots \; k)$-twisted $V^{\otimes k}$-module structure given a parity-twisted $V$-module. 

Note that above we allude to the possibility of building up  permutation-twisted $V^{\otimes k}$-modules for general permutations from the cyclically-twisted modules as constructed in this paper for even cycles and in \cite{B-superpermutation-odd} for odd cycles.  However, in the setting of vertex operator superalgebras, this patching together of $g_1 g_2$-twisted $(V_1 \otimes V_2)$-modules from $g_j$-twisted $V_j$-modules, for $j = 1,2$, has subtleties that arise which complicate the situation in comparison to the nonsuper case as handled in for instance \cite{BDM} for $g$ written as the product of disjoint cycles.  We hope to address these issues in future work.  

In addition, in this work we point out a clarification made first in \cite{BV-fermion} about the definition of $g$-twisted $V$-module for $V$ a vertex operator superalgebra.  In particular, we point out in Remark \ref{parity-stability-remark} below that the notion of ``parity-unstable $g$-twisted $V$-module" as used in, for instance, \cite{DZ}, \cite{DZ2}, \cite{DH}, arises from a notion of $g$-twisted $V$-module that is not the natural categorical definition.  In Remark \ref{parity-stability-remark}, we recall our result from \cite{BV-fermion}, showing that these so called ``parity-unstable $g$-twisted $V$-modules" always come in pairs that together form a ``parity-stable $g$-twisted $V$-module".  Thus it is more appropriate to take the definition of $g$-twisted $V$-module to be a ``parity-stable $g$-twisted $V$-module" in the language of these other works, and then ``parity-unstable $g$-twisted $V$-modules" are simply parity-unstable {\it invariant subspaces} of a (properly defined) $g$-twisted $V$-module.  This is the point of view we take in this paper.  This fact we proved in \cite{BV-fermion} concerning the nature of parity-unstable invariant subspaces of parity-stable $g$-twisted $V$-modules can be used to clarify and simplify many aspects of past works, such as \cite{DZ}, \cite{DZ2}, \cite{DH}.

This paper is organized as follows. In Section \ref{definitions-section}, we recall the definition of vertex operator superalgebra, various notions of twisted modules, and some of their properties.  In Section \ref{Delta-section}, we define the operator $\D(z)$ on a vertex operator superalgebra  $V$ following \cite{B-superpermutation-odd} where the first author generalized the analogous operator defined in \cite{BDM} to the setting of vertex operator superalgebras.  We then recall several important properties of $\D(z)$ proved in \cite{B-superpermutation-odd} which are needed in subsequent sections.   This $\D(z)$ is the main operator from which our twisted vertex operators will be built, in analogy to the nonsuper setting of \cite{BDM} and the odd order super setting of \cite{B-superpermutation-odd}, but now in conjunction with parity-twisted vertex operators acting on a parity-twisted $V$-module rather than, as in \cite{BDM} and \cite{B-superpermutation-odd}, operating in conjunction with vertex operators acting on an untwisted $V$-module.

In Section  \ref{tensor-product-setting-section}, we develop the setting for $(1 \; 2 \; \cdots \; k)$-twisted $V^{\otimes k}$-modules and study the vertex operators for a parity-twisted $V$-module modified by the orbifolding $x \rightarrow x^{1/k}$ and composing with the operator $\Delta_k(x)$.  In particular we derive the supercommutator formula for these operators showing that these operators satisfy the twisted Jacobi identity for odd vectors in $V$ if and only if $k$ is an even integer.  

In Section \ref{tensor-product-twisted-construction-section}, we use the operators to define a weak $g = (1 \; 2 \;  \cdots \; k)$-twisted $\vtk$-module structure on any weak parity-twisted $V$-module in the case when  $k$ is even.  As a result we construct a functor $T_g^k$ {}from the category of weak parity-twisted $V$-modules to the category of weak $g$-twisted $\vtk$-modules such that $T_g^k$ maps weak admissible (resp., ordinary) parity-twisted $V$-modules into weak admissible (resp., ordinary) $g$-twisted $\vtk$-modules.  In addition, $T_g^k$ preserves irreducible objects.

In Section \ref{classification-section}, we define a weak parity-twisted $V$-module structure on any weak $g$-twisted $\vtk$-module, for $V$ a vertex operator superalgebra and $g = (1 \; 2 \; \cdots k)$ for $k$ even.  In so doing, we construct a functor $U_g^k$ {}from the category of weak $g$-twisted $\vtk$-modules to the category of weak parity-twisted $V$-modules such that $T_g^k \circ U_g^k$ and $U_g^k \circ T_g^k$ are the identity functors on their respective categories.   We then use this construction and classification of $(1 \; 2\; \cdots \; k)$-twisted $V^{\otimes k}$-modules in terms of parity-twisted $V$-modules to show in Corollary \ref{graded-dimension-corollary} how the graded dimensions of $(1 \; 2\; \cdots \; k)$-twisted $V^{\otimes k}$-modules are given by the graded dimensions of parity-twisted $V$-modules under the change of variables $q \mapsto q^{1/k}$.

We have intentionally organized this paper to parallel the odd order case developed in \cite{B-superpermutation-odd} as well as the nonsuper case developed in \cite{BDM} so as to highlight the similarities and differences between these settings.

\section{Vertex operator superalgebras, twisted modules and some of their properties}\label{definitions-section}

In this section we recall some of the formal calculus we will need, and we recall the notions of vertex superalgebra and of vertex operator superalgebra, following the notational conventions of \cite{LL}.  We also recall some properties of such structures.  Then we present the notion of  $g$-twisted module for a vertex operator superalgebra and an automorphism $g$ following \cite{B-superpermutation-odd} and \cite{BV-fermion}.  We discuss some categorical aspects of this definition.  Then we briefly discuss the parity map and parity-twisted modules for a vertex operator superalgebra.  

\subsection{Formal calculus}\label{formal-calculus-section}

Let $x, x_0, x_1, x_2,$ etc., denote commuting independent formal variables.
Let $\delta (x) = \sum_{n \in \Z} x^n$.  We will use the binomial
expansion convention, namely, that any expression such as $(x_1 -
x_2)^n$ for $n \in \C$ is to be expanded as a formal power series in
nonnegative integral powers of the second variable, in this case
$x_2$.

For $r \in \mathbb{C}$ we have
\begin{equation}\label{delta-function1}
x_2^{-1}\left(\frac{x_1-x_0} {x_2}\right)^{r}
\delta\left(\frac{x_1-x_0} {x_2}\right) =
x_1^{-1}\left(\frac{x_2+x_0} {x_1}\right)^{-r}\delta
\left(\frac{x_2+x_0} {x_1}\right) ,
\end{equation}
and it is easy to see that for $k$ a positive integer,
\begin{equation}\label{delta-function2}
\sum_{p=0}^{k-1}\left(\frac{x_1-x_{0}}{x_2}\right)^{p/k}
x_2^{-1}\delta\left(\frac{x_1-x_0}{x_2}\right)=
x_2^{-1}\delta\Biggl(\frac{(x_1-x_0)^{1/k}}{x_2^{1/k}}\Biggr).
\end{equation}
Therefore, we have the $\delta$-function identity
\begin{equation}\label{delta-function3}
x_2^{-1} \delta \Biggl( \frac{(x_1 - x_0)^{1/k}}{x_2^{1/k}} \Biggr) = 
x_1^{-1} \delta \Biggl( \frac{(x_2 + x_0)^{1/k}}{x_1^{1/k}} \Biggr).
\end{equation}

We also have the three-term $\delta$-function identity 
\begin{equation}\label{three-term-delta}
x_{0}^{-1}\delta\left(\frac{x_1-x_2}{x_{0}}\right)-x_{0}^{-1}
\delta\left(\frac{x_2-x_1}{-x_{0}}\right)=x_2^{-1}\delta\left(\frac{x_1-x_0}{x_2}\right).
\end{equation}

Let $R$ be a ring, and let $O$ be an invertible linear operator on
$R[x, x^{-1}]$.  We define another linear operator $O^{\Lo}$ by
\[O^{\Lo} \cdot x^n = O^n x^n \]
for any $n \in \Z$.  For example, since the formal variable $z^{1/k}$ 
can be thought of as an invertible linear multiplication operator  
from $\C [x, x^{-1}]$ to $\mathbb{C}[z^{1/k},z^{-1/k}]
[x,x^{-1}]$, we have the corresponding operator $z^{(1/k) \Lo}$ {}from $\C 
[x,x^{-1}]$ to $\mathbb{C}[z^{1/k},z^{-1/k}]
[x,x^{-1}]$.  Note that $z^{(1/k) \Lo}$ can  be extended to a 
linear operator on $\C [[x,x^{-1}]]$ in the obvious way.

\subsection{Vertex superalgebras, vertex operator superalgebras, and some of their properties}

A {\it vertex superalgebra} is a vector space which is $\Z_2$-graded (by {\it sign} or {\it parity})
\begin{equation}
V= V^{(0)} \oplus V^{(1)}
\end{equation}
equipped with a linear map
\begin{eqnarray}
V &\longrightarrow& (\mbox{End}\,V)[[x,x^{-1}]]\\
v &\mapsto& Y(v,x)= \sum_{n\in\Z}v_nx^{-n-1} \nonumber
\end{eqnarray}
such that $v_n \in (\mathrm{End} \; V)^{(j)}$ for $v \in V^{(j)}$, $j \in \Z_2$, 
and equipped with a distinguished vector ${\bf 1} \in V^{(0)}$, (the {\it vacuum vector}), satisfying the following conditions for $u, v \in
V$:
\begin{eqnarray}
u_nv & \! = \! & 0\ \ \ \ \ \mbox{for $n$ sufficiently large};\\
Y({\bf 1},x) & \! = \! & Id_V;\\
Y(v,x){\bf 1} & \! \in \! & V[[x]]\ \ \ \mbox{and}\ \ \ \lim_{x\to 0}Y(v,x){\bf 1}\ = \ v; 
\end{eqnarray}
and for $u,v \in  V$ of homogeneous sign, the {\it Jacobi identity} holds 
\begin{equation*}
x^{-1}_0\delta\left(\frac{x_1-x_2}{x_0}\right) Y(u,x_1)Y(v,x_2) - (-1)^{|u||v|}
x^{-1}_0\delta\left(\frac{x_2-x_1}{-x_0}\right) Y(v,x_2)Y(u,x_1)
\end{equation*}
\begin{equation}
= x_2^{-1}\delta \left(\frac{x_1-x_0}{x_2}\right) Y(Y(u,x_0)v,x_2)
\end{equation}
where $|v| = j$ if $v \in V^{(j)}$ for $j \in \Z_2$.

This completes the definition. We denote the vertex superalgebra just
defined by $(V,Y,{\bf 1})$, or briefly, by $V$.

Note that as a consequence of the definition, we have that there exists a distinguished endomorphism $T \in (\mathrm{End} \; V)^{(0)}$ defined by 
\[ T(v) = v_{-2} {\bf 1} \ \ \ \ \mbox{for $v \in V$} \]
such that 
\[ [T, Y(v,x)] = Y(T(v), x)  \ = \ \frac{d}{dx} Y(v,x),\]
(cf. \cite{LL}, \cite{Barron-alternate}, \cite{Barron-n2axiomatic}).

A {\it vertex operator superalgebra} is a vertex superalgebra with a distinguished vector $\omega\in V_2$ (the {\it conformal element}) satisfying the following conditions: 
\begin{equation}
[L(m),L(n)]=(m-n)L(m+n)+\frac{1}{12}(m^3-m)\delta_{m+n,0}c
\end{equation}
for $m, n\in \Z,$ where
\begin{equation}
L(n)=\omega_{n+1}\ \ \ \mbox{for $n\in \Z$, \ \ \ \ i.e.},\
Y(\omega,x)=\sum_{n\in\Z}L(n)x^{-n-2}
\end{equation}
and $c \in \mathbb{C}$ (the {\it central charge} of $V$);
\begin{equation}
T = L(-1) \ \ \ \ \mbox{i.e.}, \ \frac{d}{dx}Y(v,x)=Y(L(-1)v,x) \ \mbox{for $v \in V$}; 
\end{equation}
$V$ is $\frac{1}{2}\Z$-graded (by {\it weight}) 
\begin{equation}
V=\coprod_{n\in \frac{1}{2} \Z}V_n 
\end{equation}
such that 
\begin{eqnarray}
L(0)v & \! = \! & nv \ = \ (\mbox{wt}\,v)v \ \ \ \mbox{for $n \in  \frac{1}{2} \Z$ and $v\in V_n$}; \\
{\rm dim} \, V_n & \! < \! & \infty ;\\
V_n & \! = \! &  0 \ \ \ \ \mbox{for $n$ sufficiently negative};
\end{eqnarray}
and $V^{(j)} = \coprod_{n\in \Z + \frac{j}{2}}V_n$ for $j \in \Z_2$.

This completes the definition. We denote the vertex operator superalgebra just
defined by $(V,Y,{\bf 1},\omega)$, or briefly, by $V$.

\begin{rem}{\em  For the purposes of this paper we do not assume any supersymmetric properties of a vertex operator superalgebra.  That is we do not assume that $V$ is necessarily a representation for any super extension of the Virasoro algebra.   However one of the main motivations for constructing and classifying permutation-twisted modules for tensor product  vertex operator superalgebras is the application to constructing mirror-twisted sectors for N=2 supersymmetric vertex operator superalgebras as discussed in \cite{Barron-varna}, \cite{Barron-n2twisted}, and as presented as an application to the results of this paper in \cite{B-varna2013}. 
}
\end{rem}

\begin{rem}\label{VOSAs-tensor-remark} {\em Note that if $(V, Y, {\bf 1})$ and $(V', Y', {\bf 1}')$ are two vertex superalgebras, then $(V \otimes V', \; Y \otimes Y', \; {\bf 1} \otimes {\bf 1}')$ is a vertex superalgebra where
\begin{equation}\label{define-tensor-product}
(Y \otimes Y') (u \otimes u', x) (v \otimes v') = (-1)^{|u'||v|} Y(u,x)v \otimes Y'(u',x)v'.
\end{equation}
If in addition, $V$ and $V'$ are vertex operator superalgebras with conformal vectors $\omega$ and $\omega'$ respectively, then $V\otimes V'$ is a vertex operator superalgebra with conformal vector $\omega \otimes {\bf 1}' + {\bf 1} \otimes \omega'$.}
\end{rem}

\begin{rem}\label{parity-grading-on-V}
{\em
As a consequence of the definition of vertex operator superalgebra, independent of the requirement that as a vertex superalgebra we should have $v_n \in (\mathrm{End} \, V)^{(|v|)}$, we have that $\mathrm{wt} (v_n u ) = \mathrm{wt} u + \mathrm{wt} v - n -1$, for $u,v \in V$ and $n \in \Z$.  This implies that $v_n \in (\mathrm{End} \, V)^{(j)}$ if and only if $v \in V^{(j)}$ for $j \in \mathbb{Z}_2$, without us having to assume this as an axiom.  
}
\end{rem}

\subsection{The notion of twisted module}\label{twisted-module-definition-section}

Let $(V, Y, {\bf 1})$ and $(V', Y', {\bf 1}')$ be vertex superalgebras.  A {\it homomorphism of vertex superalgebras} is a linear map $g: V \longrightarrow V'$ of $\Z_2$-graded vector spaces  such that $g({\bf 1}) = {\bf 1}'$ and 
\begin{equation}\label{automorphism}
g Y(v,x) =Y'(gv,x)g
\end{equation}
for $v\in V.$   Note that this implies that $g \circ T = T'\circ g$.  If in addition, $V$ and $V'$ are vertex operator superalgebras with conformal elements $\omega$ and $\omega'$, respectively, then a  {\it homomorphism of vertex operator superalgebras} is a homomorphism of vertex superalgebras $g$ such that $g(\omega) = \omega'$.  In particular $g V_n\subset V'_n$ for $n\in  \frac{1}{2} \mathbb{Z}$.

An {\it automorphism} of a vertex (operator) superalgebra $V$ is a bijective vertex (operator) superalgebra homomorphism from $V$ to $V$.

If $g$ is an automorphism of a vertex (operator) superalgebra $V$ such that $g$ 
has finite order, then $V$ is a direct sum of the eigenspaces $V^j$ of $g$,
\begin{equation}
V=\coprod_{j\in \Z /k \Z }V^j,
\end{equation}
where $k \in \mathbb{Z}_+$ and $g^k = 1$, and
\begin{equation}
V^j=\{v\in V \; | \; g v= \eta^j v\},
\end{equation}
for $\eta$ a fixed primitive $k$-th root of unity.  We denote the projection of $v \in V$ onto the $j$-th eigenspace, $V^j$, by $v_{(j)}$.

Let $(V, Y, \mathbf{1})$ be a vertex superalgebra and $g$ an automorphism of $V$ of period $k$.  A {\em $g$-twisted $V$-module} is a $\mathbb{Z}_2$-graded vector space $M = M^{(0)} \oplus M^{(1)}$ equipped with a linear map 
\begin{eqnarray}
V &\rightarrow& (\End \, M)[[x^{1/k},x^{-1/k}]] \\  
v  &\mapsto& Y_g(v,x)=\sum_{n\in \frac{1}{k} \mathbb{Z} }v^g_n x^{-n-1}, \nonumber 
\end{eqnarray}
with $v_n^g \in (\mathrm{End} \; M)^{(|v|)}$, such that for $u,v\in V$ and $w\in M$ the following hold: 
\begin{equation}
v^g_nw=0 \mbox{ if $n$ is sufficiently large};
\end{equation}
\begin{equation} 
Y_g({\bf 1},x)= Id_M;
\end{equation}
the twisted Jacobi identity holds: for $u, v\in V$ of homogeneous sign
\[x^{-1}_0\delta\left(\frac{x_1-x_2}{x_0}\right)
Y_g(u,x_1)Y_g(v,x_2)- (-1)^{|u||v|} x^{-1}_0\delta\left(\frac{x_2-x_1}{-x_0}\right)
Y_g(v,x_2)Y_g(u,x_1) \]
\begin{equation}\label{twisted-Jacobi}
= \frac{x_2^{-1}}{k} \sum_{ j \in \mathbb{Z}/k\mathbb{Z}} 
\delta\left( \eta^j \frac{(x_1-x_0)^{1/k}}{x_2^{1/k}}\right) Y_g(Y(g^ju,x_0)v,x_2).
\end{equation}

This completes the definition of $g$-twisted $V$-module for a vertex superalgebra.  We denote the  $g$-twisted $V$-module just defined by $(M, Y_g)$ or just by $M$ for short.  

We note that the generalized twisted Jacobi identity (\ref{twisted-Jacobi}) is equivalent to
\[x^{-1}_0\delta\left(\frac{x_1-x_2}{x_0}\right)
Y_g(u,x_1)Y_g(v,x_2)- (-1)^{|u||v|} x^{-1}_0\delta\left(\frac{x_2-x_1}{-x_0}\right)
Y_g(v,x_2)Y_g(u,x_1) \]
\begin{equation}\label{twisted-Jacobi-eigenspace}
= x_2^{-1}  \left(\frac{x_1-x_0}{x_2}\right)^{-r/k}
\delta\left( \frac{x_1-x_0}{x_2}\right) Y_g(Y(u,x_0)v,x_2)
\end{equation}
for $u \in V^r$, $r = 0, \dots, k-1$.  In addition, this implies that for $v\in V^r$, 
\begin{equation}\label{Y-for-eigenvector}
Y_g(v,x)=\sum_{n\in r/k + \Z }v^g_n x^{-n-1},
\end{equation}
and that 
\begin{equation}\label{limit-axiom}
Y_g(gv,x) = \lim_{x^{1/k} \to \eta^{-1} x^{1/k}} Y_g(v,x) 
\end{equation}
for $v \in V$.

Note that if $g = 1$, then a $g$-twisted $V$-module is a $V$-module for the vertex superalgebra $V$.  If $(V, Y, {\bf 1}, \omega)$ is a vertex operator superalgebra and $g$ is a vertex operator superalgebra automorphism of $V$, then since $\omega \in V^{0}$, we have that $Y_g(\o,x)$ has component operators which satisfy the Virasoro algebra relations and $Y_g(L(-1)u,x)=\frac{d}{dx}Y_g(u,x)$. In this case, a $g$-twisted $V$-module as defined above, viewed as a vertex superalgebra module, is called a {\em weak $g$-twisted $V$-module} for the vertex operator superalgebra $V$.

A {\em weak admissible}  $g$-twisted $V$-module is a weak $g$-twisted
$V$-module $M$ which carries a $\frac{1}{2k}\N$-grading
\begin{equation}\label{m3.12}
M=\coprod_{n\in\frac{1}{2k}\N}M(n)
\end{equation}
such that $v^g_mM(n)\subseteq M(n+\wt \; v-m-1)$ for homogeneous $v\in V$, $n \in \frac{1}{2k} \N$, and $m \in \frac{1}{k} \Z$.  We may assume that $M(0)\ne 0$ if $M\ne 0$.   If $g = 1$, then a weak admissible $g$-twisted $V$-module is called a weak admissible $V$-module.  

\begin{rem}\label{even-grading-remark}
{\em
Note that if $k$ is even where $k$ is the order of $g$, then the grading of a weak admissible $g$-twisted $V$-module can be assumed to be a $\frac{1}{k}\mathbb{N}$ grading. 
}
\end{rem}

An (ordinary) $g$-twisted $V$-module is a weak $g$-twisted $V$-module $M$ graded by $\C$ induced by the spectrum of $L(0).$ That is, we have
\begin{equation}\label{g3.14}
M=\coprod_{\lambda \in{\C}}M_{\lambda} 
\end{equation}
where $M_{\l}=\{w\in M|L(0)^gw=\l w\}$, for $L(0)^g = \omega_1^g$. Moreover we require that $\dim
M_{\l}$ is finite and $M_{n/2k +\l}=0$ for fixed $\l$ and for all
sufficiently small integers $n$.  If $g = 1$, then a $g$-twisted $V$-module is a $V$-module.

A {\it homomorphism of weak $g$-twisted $V$-modules}, $(M, Y_g)$ and $(M', Y_g')$, is a linear map $f: M \longrightarrow M'$ satisfying 
\begin{equation}
f(Y_g(v, x)w) = Y'_g(v, x) f(w).
\end{equation} 
for $v \in V$, and $w \in M$.  If in addition, $M$ and $M'$ are weak admissible $g$-twisted $V$-modules, then a {\it homomorphism of weak admissible $g$-twisted $V$-modules}, is a homomorphism of weak $g$-twisted $V$-modules such that 
$f(M(n)) \subseteq M'(n)$.  And if $M$ and $M'$ are ordinary $g$-twisted $V$-modules, then a {\it homomorphism of $g$-twisted $V$-modules}, is a homomorphism of weak  $g$-twisted $V$-modules such that 
$f(M_\lambda) \subseteq M'_\lambda$. 

We note here that an example of an automorphism of a vertex operator superalgebra is the {\it parity map}
\begin{eqnarray}\label{define-parity}
\sigma : V &\longrightarrow& V\\
v & \mapsto & (-1)^{|v|} v . \nonumber
\end{eqnarray}

\begin{rem}\label{parity-stability-remark}
{\em
In many works on vertex superalgebras, e.g. \cite{Li2}, \cite{DZ}, \cite{DZ2}, \cite{DH}, \cite{Barron-varna}, \cite{Barron-n2twisted},  the condition that $v_n^g \in (\mathrm{End} \; M)^{(|v|)}$ for $v \in V$, is not given as one of the axioms of a $g$-twisted $V$-module $M$ for a vertex superalgebra $V$.  That is, it is not assumed that the $\mathbb{Z}_2$-grading of $V$ is compatible with the $\mathbb{Z}_2$-grading of $M$ via the action of $V$ as super-endomorphisms acting on $M$.  Then in, for instance, \cite{DZ}, \cite{DZ2}, \cite{DH}, the notion of ``parity-stable $g$-twisted $V$-module" is introduced for those modules that are representative of the $\mathbb{Z}_2$-grading of $V$, and modules that do not have this property are called ``parity-unstable".  Thus a ``parity-unstable $g$-twisted $V$-module" is a vector space $M$ that satisfies all the axioms of our notion of $g$-twisted $V$-module except for $v_n^g \in (\mathrm{End} \; M)^{(|v|)}$.  That is there exists no $\mathbb{Z}_2$-grading on $M$ such that the operators $v_n^g$ act as even or odd endomorphisms on $M$ according to the sign (or parity) of $v$.   However in \cite{BV-fermion}, we prove that any so called ``parity-unstable $g$-twisted $V$-module" can always be realized as a subspace of a $g$-twisted $V$-module in the sense of the notion of $g$-twisted $V$-module we give above.  In particular, in \cite{BV-fermion} we proved the following (reworded to fit our current setting): 
\vspace{-.2in}
\begin{quote}
\begin{thm}\label{parity-stability-theorem}(\cite{BV-fermion})
Let $V$ be a vertex superalgebra and $g$ an automorphism.  Suppose $(M, Y_M)$ is a ``parity-unstable $g$-twisted $V$-module" (in the sense of \cite{DZ}).  Then $(M, Y_M \circ \sigma_V)$ is a ``parity-unstable $g$-twisted $V$-module" which is not isomorphic to $(M, Y_M)$.  Moreover $(M, Y_M)  \oplus (M, Y_M \circ \sigma_V)$ is a ``parity-stable $g$-twisted $V$-module" (in the sense of \cite{DZ}), i.e., a $g$-twisted $V$ module in terms of the definition given above in this paper.   
\end{thm}
\end{quote}
Requiring weak twisted modules to be ``parity stable" as part of the definition gives the more canonical notion of twisted module from a categorical point of view, for instance, for the purpose of defining a $(V_1 \otimes V_2)$-module structure on $M_1 \otimes M_2$ for $M_j$ a $V_j$-module, for $j = 1,2$.  (See e.g., (\ref{define-tensor-product})).  In particular, the notion of a $g$-twisted $V$-module corresponding to a {\it representation} of $V$ as a vertex superalgebra only holds for ``parity-stable $g$-twisted $V$-modules", in that the vertex operators acting on a $g$-twisted $V$-module have coefficients in $\mathrm{End} \, M $ such that, the operators $v_m^g$ have a $\mathbb{Z}_2$-graded structure compatible with that of $V$.   For instance the operators $v_0^g$, for $v \in V$, give a representation of the Lie superalgebra generated by $v_0$ in $\mathrm{End} \, V$ if and only if $M$ is ``parity stable".  This corresponds to $V$ acting (via the modes of the vertex operators) as endomorphisms on $M$ in the category of vector spaces (i.e., via even or odd endomorphisms) rather than in the category of $\mathbb{Z}_2$-graded vectors spaces (i.e., as grade-preserving and thus strictly even endomorphisms).   However, it is interesting to note that, as is shown in \cite{BV-fermion}, for a lift of a lattice isometry, the ``twisted modules" for a lattice vertex operator superalgebra constructed following \cite{DL2}, \cite{Xu}, naturally sometimes give rise to pairs of parity-unstable invariant subspaces in the language of the current paper, i.e., to pairs of ``parity-unstable $g$-twisted modules", that then must be taken as a direct sum to realize the actual $g$-twisted module that is constructed.  }
\end{rem}

\subsection{Parity-twisted $V$-modules}  A crucial example in the study of $g$-twisted $V$-modules for $V$ a vertex superalgebra is that of parity-twisted $V$-modules.  (Not to be confused with the notion discussed above of ``parity-stable" or ``parity-unstable" modules.) Above in (\ref{define-parity}), we define the parity automorphism, denoted $\sigma$, for any vertex superalgebra.  Thus we have the notion of a {\it parity-twisted $V$-module}, also denoted by {\it $\sigma$-twisted $V$-module}.

\begin{rem}\label{parity-grading-remark}
{\em Note that it follows from the definitions, that any weak admissible $\sigma$-twisted module for a vertex operator superalgebra is $\mathbb{N}$-graded.  }
\end{rem}

If $V$,  in addition to being a vertex operator superalgebra is N=1 or N=2 supersymmetric, i.e., is also a representation of the N=1 or N=2 Neveu-Schwarz algebra super extension of the Virasoro algebra, then a $\sigma$-twisted $V$-module is naturally a representation of the N=1 or N=2 Ramond algebra, respectively, cf. \cite{Barron-varna}, \cite{Barron-n2twisted}, \cite{B-varna2013} and references therein.

\section{The operator $\Delta_k (x)$}\label{Delta-section}
\setcounter{equation}{0} 

In this section, we recall the operator $\Delta_k(x)$ on a vertex operator superalgebra $V$ for a fixed positive integer $k$ as first defined in \cite{BDM} and then extended to vertex operator superalgebras in \cite{B-superpermutation-odd}.  In Section \ref{tensor-product-twisted-construction-section}, we will use $\Delta_k(x)$ for $k$ even to construct a $(1 \; 2 \; \cdots \; k)$-twisted  $V^{\otimes k}$-module {}from a parity-twisted $V$-module.

Let $\ZZ$ denote the positive integers.  Let $x$, $z$, and $z_0$ be formal variables commuting with each other.  Consider the polynomial
\[\frac{1}{k} (1 + x)^k - \frac{1}{k} \in x \C [x] . \]  
Following \cite{B-superpermutation-odd}, for $k \in \ZZ$, we define $a_j \in \C$ for $j \in \ZZ$, by  
\begin{equation}\label{define-a}
\exp \Biggl( - \sum_{j \in \ZZ} a_j  \Lx
\Biggr) \cdot x = \frac{1}{k} (1 + x)^k -
\frac{1}{k} .
\end{equation}
For example, $a_1=(1-k)/2$ and $a_2=(k^2-1)/12.$

Let
\begin{eqnarray*}
f(x) &=& z^{1/k} \exp \Biggl(- \sum_{j \in \ZZ} a_j  \Lx \Biggr) 
\cdot x \\
&=& \exp \Biggl(- \sum_{j \in \ZZ} a_j  \Lx \Biggr) \cdot
z^{(1/k) \Lo} \cdot x\\
&=& \frac{z^{1/k}}{k} (1 + x)^k - \frac{z^{1/k}}{k} \; \; \in
z^{1/k}x\mathbb{C}[x].
\end{eqnarray*}
Then the compositional inverse of $f(x)$ in $x\mathbb{C}[z^{-1/k}, 
z^{1/k}][[x]]$ is given by
\begin{eqnarray*}
f^{-1} (x) &=& z^{- (1/k) \Lo} \exp \Biggl( \sum_{j \in \ZZ} a_j  \Lx 
\Biggr) \cdot x \\
&=& z^{- 1/k} \exp \Biggl( \sum_{j \in \ZZ} a_j 
z^{- j/k} \Lx \Biggr) \cdot x \\
&=& (1 + k z^{- 1/k} x)^{1/k} - 1
\end{eqnarray*} 
where the last line is considered as a formal power series in
$z^{-1/k}x \mathbb{C}[z^{-1/k}][[x]]$, i.e., we are expanding
about $x = 0$ taking $1^{1/k} = 1$.

Let $V= (V, Y, {\bf 1}, \omega)$ be a vertex operator superalgebra.  In 
$(\End \;V)[[z^{1/2k}, z^{-1/2k}]]$, define 
\begin{equation}\label{Delta-for-a-module}
\Delta_k (z) = \exp \Biggl( \sum_{j \in \ZZ} a_j z^{- \frac{j}{k}} L(j) 
\Biggr) (k^\frac{1}{2})^{-2L(0)} \left(z^{\frac{1}{2k}\left( k-1\right)}\right)^{- 2L(0)} .
\end{equation}
In \cite{B-superpermutation-odd}, we proved the following proposition and lemma.

\begin{prop}\label{psun1} (\cite{B-superpermutation-odd})
Let $V$ be a vertex operator superalgebra.  We have the following identity in  $((\End \;V)[[z^{1/2k}, z^{-1/2k}]]) [[z_0, z_0^{-1}]]$ 
\[\Delta_k (z) Y(u, z_0) \Delta_k (z)^{-1} = Y(\Delta_k 
(z + z_0)u, \left( z + z_0 \right)^{1/k} 
- z^{1/k}  ) ,\]
for all $u \in V$.  
\end{prop}

\begin{lem}\label{c2.5} (\cite{B-superpermutation-odd})
For $V$ a vertex operator superalgebra, in $(\End \; V)[[z^{1/2k}, z^{-1/2k}]]$, we have
\begin{eqnarray}
\Delta_k (z) L(-1) - \frac{1}{k} z^{1/k - 1} L(-1) 
\Delta_k (z) 
&=&  \frac{\partial}{\partial z}\Delta_k (z) , \label{identity in voa}\\ 
\Delta_k (z)^{-1} L(-1) - k z^{- 1/k + 1} L(-1) 
\Delta_k (z)^{-1}   
&=&  k z^{- 1/k + 1} \frac{\partial}{\partial z}
\Delta_k (z)^{-1}  . \label{second identity in voa}
\end{eqnarray}
\end{lem}

\section{The setting of $(1 \; 2 \; \cdots k$)-twisted $V^{\otimes k}$-modules and the operators $Y_\sigma(\Delta_k(x)u, x^{1/k})$ for a $\sigma$-twisted $V$-module $(M_\sigma, Y_\sigma)$}\label{tensor-product-setting-section}
\setcounter{equation}{0}

Now we turn our attention to tensor product vertex operator superalgebras.
Let $V=(V,Y,{\bf 1},\omega)$ be a vertex operator superalgebra, and let $k$ be a
fixed positive integer.  Then by Remark \ref{VOSAs-tensor-remark},  $V^{\otimes k}$ is also a
vertex operator superalgebra, and the permutation group $S_k$
acts naturally on $\vtk$ as signed automorphisms.  That is
$(j \; j+1) \cdot (v_1 \otimes v_2 \otimes \cdots \otimes v_k) = (-1)^{|v_j||v_{j+1}|} (v_1 \otimes v_2 \otimes \cdots v_{j-1} \otimes v_{j+1} \otimes v_j \otimes v_{j+2} \otimes \cdots \otimes v_k)$, and we take this to be a right action so that, for instance
\begin{eqnarray}
\qquad (1 \; 2  \cdots  k) : V \otimes V \otimes \cdots \otimes V \! \! \! \! &\longrightarrow & \! \! \! \! V \otimes V \otimes \cdots \otimes V\\
v_1 \otimes v_2 \otimes \cdots \otimes v_k \! \! \! \! \! \! & \mapsto & \! \! \! \! \!  \! (-1)^{|v_1|(|v_2| + \cdots + |v_k|)} v_2 \otimes v_3 \otimes \cdots \otimes v_k \otimes v_1. \nonumber
\end{eqnarray}
(Note that in \cite{BDM}, this action was given as a left action.  For convenience, we make the change here to a right action as in \cite{BHL} and \cite{B-superpermutation-odd}).

Let $g=(1 \; 2 \; \cdots \; k)$.  In the next section, we will construct a functor $T_g^k$ {}from the category of weak $\sigma$-twisted $V$-modules to the category of weak $g$-twisted modules for $\vtk$ for the case when $k$ is even.  This construction will be based on the operators $Y_\sigma(\Delta_k(x)u, x^{1/k})$ for a parity-twisted $V$-module $(M_\sigma, Y_\sigma)$.  Thus in this section, we establish several properties of these operators.

For $v\in V$, and $k$ any positive integer, denote by $v^j\in\vtk$, for $j = 1,\dots, k$, the vector whose $j$-th tensor factor is $v$ and whose other tensor factors are ${\bf 1}$.  Then
$gv^j=v^{j-1}$ for $j=1,\dots,k$ where $0$ is understood to be $k$. 

Suppose that $W$ is a weak $g$-twisted $\vtk$-module, and let $\eta$ be a fixed primitive $k$-th root of unity.  We first make some general observations for this setting following \cite{BDM} and \cite{B-superpermutation-odd}.  First, it follows {}from the definition of twisted module (cf. (\ref{limit-axiom})) that
the $g$-twisted vertex operators on $W$ satisfy
\begin{equation}
Y_g(v^{j+1},x) =  Y_g(g^{-j} v^1,x) = \lim_{x^{1/k}\to \eta^{j}x^{1/k}}  Y_g(v^1,x).
\end{equation}
Since $\vtk$ is generated by $v^j$ for $v\in V$ and $j=1,...,k,$ 
the twisted vertex operators $Y_g(v^1,x)$ for $v\in V$ determine all
the twisted vertex operators $Y_g(u,x)$ on $W$ for any $u\in\vtk$. This
observation is very important in our construction of twisted modules.

Secondly, if $u,v\in V$ are of homogeneous sign, then by (\ref{twisted-Jacobi}) the twisted Jacobi identity for $Y_g(u^1,x_1)$ and $Y_g(v^1,x_2)$ is
\begin{multline}\label{k1}
x^{-1}_0\delta\left(\frac{x_1-x_2}{x_0}\right)
Y_g(u^1,x_1)Y_g(v^1,x_2)\\
-(-1)^{ |u||v|} x^{-1}_0\delta\left(\frac{x_2-x_1}{-x_0}\right)
Y_g(v^1,x_2)Y_g(u^1,x_1)\\
=\frac{1}{k}x_2^{-1}\sum_{j=0}^{k-1}\delta\Biggl(\eta^j\frac{(x_1-x_0)^
{1/k}}{x_2^{1/k}}\Biggr)Y_g(Y(g^ju^1,x_0)v^1,x_2).
\end{multline} 
Since $g^{-j}u^1=u^{j+1}$, we see that $Y(g^{-j}u^1,x_0)v^1$ only involves nonnegative integer powers of $x_0$ unless $j=0\ (\mod \; k).$   Thus the we have the supercommutator
\begin{multline}\label{k2}
[Y_g(u^1,x_1),Y_g(v^1,x_2)] \\
= \; \Res_{x_0}\frac{1}{k}x_2^{-1}\delta
\Biggl(\frac{(x_1-x_0)^{1/k}}{x_2^{1/k}}\Biggr)Y_g(Y(u^1,x_0)v^1,x_2) .
\end{multline}
This shows that the component operators of $Y_g(u^1,x)$ for $u\in V$ on
$W$ form a Lie superalgebra.

For $u\in V$ and $\D(z)$ given by (\ref{Delta-for-a-module}), and $(M_\sigma, Y_\sigma)$ a parity-twisted $V$-module, define
\begin{equation}
\bar{Y}_\sigma (u,x)=Y_\sigma (\D(x)u,x^{1/k}) .
\end{equation}
For example, taking $u=\omega$, and recalling that $a_2= (k^2-1)/12$, we have
\begin{eqnarray}\label{sun1}
\bY(\omega,x) &=&Y_\sigma \left(\frac{x^{2(1/k - 1)}}{k^2}\Bigl(\omega + a_2
\frac{c}{2}x^{-2/k} \Bigr),x^{1/k}\right) \\
&=&  \frac{x^{2(1/k - 1)}}{k^2}Y_\sigma (\omega,x^{1/k})
+\frac{(k^2-1)c}{24k^2}x^{-2} \nonumber
\end{eqnarray}
where $c$ is the central charge of $V$.

\begin{rem}\label{wrong-space-remark}{\em  Since 
$Y_\sigma (v, x) \in  x^{|v|/2} (\mathrm{End}\, M_\sigma) [[x, x^{-1}]]$,
and for $k$ even 
\begin{equation}
\D(x)u \in 
\left\{ \begin{array}{ll}
 V^{(0)} [[x^{1/k}, x^{-1/k}]] & \mbox{if $u$ is even}\\
\\
x^{1/2k} V^{(1)} [[x^{1/k}, x^{-1/k}]] & \mbox{ if $u$ is odd}
\end{array}
\right. ,
\end{equation}
we have that if $k$ is even
\begin{equation}
\bar{Y}_\sigma (u,x) = Y_\sigma (\D(x)u,x^{1/k}) \in (\mathrm{End}\, M_\sigma) [[x^{1/k}, x^{-1/k}]] .
\end{equation}
When we put a weak $g$-twisted $\vtk$-module structure on $M_\sigma$, this operator 
$\bar{Y}_\sigma (u,x)$ will be the twisted vertex operator acting on $M_\sigma$ associated to $u^1$, where we will assume $k$ is even.  Note however, that if $k$ is odd, then for $u$ odd in $V$, we have $Y_\sigma (\D(x)u,x^{1/k}) \in (\mathrm{End}\, M_\sigma) [[x^{1/2k}, x^{-1/2k}]]$.  This is a reflection of the $k$ odd case as constructed and classified in \cite{B-superpermutation-odd} being fundamentally different from the $k$ even case.  
}
\end{rem}

Next we study the properties of the operators $\bY(u,x)$, following and generalizing \cite{BDM} and \cite{B-superpermutation-odd}.

\begin{lem}\label{l3.1} For $u\in V$
\[\bY(L(-1)u,x)=\frac{d}{dx}\bY(u,x).\]
\end{lem}

\begin{proof} By Lemma \ref{c2.5},  and the $L^\sigma(-1)$-derivative property for $(M_\sigma, Y_\sigma)$, we have
\begin{eqnarray*}
\bY(L(-1)u,x) &=& Y_\sigma(\D(x)L(-1)u,x^{1/k})\\
&=& Y_\sigma(\frac{d}{dx}\D(x)u,x^{1/k}) + k^{-1} x^{1/k - 1}Y_\sigma(L(-1)\D(x)u,x^{1/k})\\
&=& Y_\sigma(\frac{d}{dx}\D(z)u,x^{1/k}) + \left. k^{-1} x^{1/k - 1}\frac{d}{dy}
Y_\sigma(\D(x)u,y) \right|_{y=x^{1/k}}\\
&=& Y_\sigma(\frac{d}{dx}\D(x)u,x^{1/k}) + \left.\frac{d}{dy}Y_\sigma(\D(x)u,y^{1/k})
\right|_{y=x}\\
&=& \frac{d}{dx}Y_\sigma(\D(x)u,x^{1/k})\\
&=& \frac{d}{dx}\bY(u,x)
\end{eqnarray*}
as desired. \end{proof}

\begin{lem}\label{l3.2} For $u,v\in V$ of homogeneous sign, we have the supercommutator
\begin{multline}\label{first-supercommutator}
 [\bY(u,x_1),\bY(v,x_2)] \\
= \left\{ \begin{array}{ll}
\Res_{x_0}\frac{x_2^{-1}}{k}
\delta\Biggl(\frac{(x_1-x_0)^{1/k}}{x_2^{1/k}}\Biggr)\bY(Y(u,x_0)v,x_2) & \mbox{if $k$ is even}\\
\Res_{x_0} \frac{x_2^{-1}}{k}
\delta\Biggl(\frac{(x_1-x_0)^{1/k}}{x_2^{1/k}}\Biggr)\bY(Y(u,x_0)v,x_2) \left( \frac{x_1 - x_0}{x_2} \right)^{ \frac{|u|}{2k}} & \mbox{if $k$ is odd}
\end{array} . \right.
\end{multline}
\end{lem}

\begin{proof} The supercommutator formula for the weak $\sigma$-twisted $V$-module $M_\sigma$ is given by 
\begin{multline}\label{commutator for Lemma 3.3}
[Y_\sigma(u,x_1),Y_\sigma(v,x_2)]  \\
= \Res_{x}x_2^{-1}\left(\frac{x_1-x}{x_2}
\right)^{-|u|/2} \delta\left(\frac{x_1-x}{x_2}
\right)Y_\sigma(Y(u,x)v,x_2)
\end{multline}
which is a consequence of the twisted Jacobi identity on $M_\sigma$, for $u$ of homogeneous sign (parity) in $V$.  Replacing $Y_\sigma(u,x_1)$ and $Y_\sigma (v,x_2)$ by $Y_\sigma (\D(x_1)u,x_1^{1/k})$ and 
$Y_\sigma(\D(x_2)v,x_2^{1/k})$, respectively, in the supercommutator
formula, we have the supercommutator
\begin{multline}\label{substitution equation}
[\bar{Y}_\sigma(u,x_1),\bar{Y}_\sigma(v,x_2)]  \\
= \Res_{x}x_2^{-1/k} \Biggl( \frac{x_1^{1/k}-x}{x_2^{1/k}}\Biggr)^{-|u|/2}
\delta\Biggl( \frac{x_1^{1/k}-x}{x_2^{1/k}}\Biggr)
Y_\sigma (Y(\D(x_1)u,x)\D(x_2)v,x_2^{1/k}) .
\end{multline}

We want to make the change of variable $x = x_1^{1/k}-(x_1-x_0)^{1/k}$
where by $x_1^{1/k}-(x_1-x_0)^{1/k}$ we mean the power series expansion
in positive powers of $x_0$.  For $n \in \mathbb{Z}$, it was shown in \cite{BDM} that
\begin{equation}
\left. (x_1^{1/k} - x)^n \right|_{x = x_1^{1/k}-(x_1-x_0)^{1/k}} = (x_1 - x_0)^{n/k}.
\end{equation}

Thus substituting $x =  x_1^{1/k}-(x_1-x_0)^{1/k}$ into 
\[ x_2^{-1/k} \Biggl( \frac{x_1^{1/k}-x}{x_2^{1/k}}\Biggr)^{-|u|/2}
\delta\Biggl( \frac{x_1^{1/k}-x}{x_2^{1/k}}\Biggr)
Y_\sigma (Y(\D(x_1)u,x)\D(x_2)v,x_2^{1/k}) ,\]
we have a well-defined power series given by
\begin{multline*}
x_2^{-1/k} \Biggl(\frac{(x_1-x_0)^{1/k}}{x_2^{1/k}}\Biggr)^{-|u|/2} \delta\Biggl(\frac{(x_1-x_0)^{1/k}}{x_2^{1/k}}\Biggr)\\
 Y_\sigma (Y(\D(x_1)u, x_1^{1/k}-(x_1-x_0)^{1/k})\D(x_2)v,x_2^{1/k}) .
\end{multline*}

Let $f(z_1,z_2,x)$ be a complex analytic function in $z_1, z_2$, and
$x$, and let $h(z_1,z_2,z_0)$ be a complex analytic function in $z_1,
z_2$, and $z_0$. Then if $f(z_1,z_2,h(z_1,z_2,z_0))$ is well defined,
and thinking of $z_1$ and $z_2$ as fixed ( i.e., considering 
$f(z_1,z_2,h(z_1,z_2,z_0))$ as a Laurent series in $z_0$) by the
residue theorem of complex analysis, we have
\begin{equation}\label{residue change of variables}
\Res_x f(z_1,z_2,x)=\Res_{z_0} \left( \frac{\partial}{\partial z_0}
h(z_1,z_2,z_0) \right)f(z_1,z_2,h(z_1,z_2,z_0)) .
\end{equation} 
This course remains true for $f$ and $h$ formal power series in
their respective variables.  Thus making the change of variable
$x= h(x_1,x_2,x_0) = x_1^{1/k}-(x_1-x_0)^{1/k}$, using (\ref{substitution
equation}), (\ref{residue change of variables}), and the 
$\delta$-function identity (\ref{delta-function3}), we obtain
\begin{eqnarray*}
\lefteqn{ [\bY(u,x_1),\bY(v,x_2)] = }\\  
&=& \Res_{x_0}\frac{1}{k}x_2^{-1/k} (x_1-x_0)^{1/k-1} \Biggl(\frac{(x_1-x_0)^{1/k}}{x_2^{1/k}}\Biggr)^{-|u|/2} \delta\Biggl(
\frac{(x_1-x_0)^{1/k}}{x_2^{1/k}}\Biggr) \\
& & \quad Y_\sigma (Y(\D(x_1)u,x_1^{1/k}-(x_1-x_0)^{1/k})
\D(x_2)v,x_2^{1/k})\\
&=& \Res_{x_0}\frac{1}{k} x_2^{-1} \Biggl(\frac{x_1-x_0}{x_2}\Biggr)^{-|u|/2k}   \delta\Biggl(
\frac{(x_1-x_0)^{1/k}}{x_2^{1/k}}\Biggr)  \\
& & \quad 
Y_\sigma (Y(\D(x_1)u,x_1^{1/k}-(x_1-x_0)^{1/k})\D(x_2)v,x_2^{1/k})\\
&=&  \Res_{x_0}\frac{1}{k}x_1^{-1} \Biggl(\frac{x_2+x_0}{x_1}\Biggr)^{|u|/2k}  
\delta\Biggl(
\frac{(x_2+x_0)^{1/k}}{x_1^{1/k}}\Biggr) \\
& & \quad 
Y_\sigma (Y(\D(x_1)u,x_1^{1/k}-(x_1-x_0)^{1/k})\D(x_2)v,x_2^{1/k}).
\end{eqnarray*}
Now we observe that  
\begin{multline}
Y_\sigma (Y(\D(x_1)u,x_1^{1/k}-(x_1-x_0)^{1/k})\D(x_2)v,x_2^{1/k}) \\
\in 
\left\{ \begin{array}{ll}
x_1^{|u|/2k} (\mathrm{End} \, M_\sigma)  [[x_0]] [[x_1^{1/k}, x_1^{-1/k}]][[x_2^{1/2k}, x_2^{-1/2k}]] & \mbox{if $k$ is even}\\
\\
(\mathrm{End} \, M_\sigma)  [[x_0]] [[x_1^{1/k}, x_1^{-1/k}]][[x_2^{1/2k}, x_2^{-1/2k}]] & \mbox{if $k$ is odd} 
\end{array}
\right. .
\end{multline}
Thus letting $p(k) = 0$ if $k$ is even and $p(k) = 1$ if $k$ is odd, using the $\delta$-function substitution property (see e.g., \cite{LL}) and Proposition \ref{psun1}, we obtain
\begin{eqnarray*}
\lefteqn{ [\bY(u,x_1),\bY(v,x_2)] = }\\  
&=& \Res_{x_0}\frac{1}{k}x_1^{-1}  \Biggl(\frac{x_2+x_0}{x_1}\Biggr)^{p(k) |u|/2k}  \delta\Biggl(
\frac{(x_2+x_0)^{1/k}}{x_1^{1/k}}\Biggr)   \\
& & \quad  Y_\sigma (Y(\D(x_2+x_0)u,(x_2+x_0)^{1/k}-x_2^{1/k})
\D(x_2)v,x_2^{1/k})\\
&=& \Res_{x_0}\frac{1}{k}x_2^{-1} \Biggl(\frac{x_1-x_0}{x_2}\Biggr)^{-p(k) |u|/2k}  \delta\Biggl(
\frac{(x_1-x_0)^{1/k}}{x_2^{1/k}}\Biggr)  \\
& & \quad Y_\sigma (Y(\D(x_2+x_0)u,(x_2+x_0)^{1/k}-x_2^{1/k})
\D(x_2)v,x_2^{1/k})\\
&=&  \Res_{x_0}\frac{1}{k}x_2^{-1} \Biggl(\frac{x_1-x_0}{x_2}\Biggr)^{-p(k) |u|/2k} \delta\Biggl(\frac{(x_1-x_0)^
{1/k}}{x_2^{1/k}}\Biggr)\\
& & \quad Y_\sigma (\D(x_2)Y(u,x_0)v,x_2^{1/k}) \\
&=& \Res_{x_0}\frac{1}{k}x_2^{-1} \Biggl(\frac{x_1-x_0}{x_2}\Biggr)^{-p(k) |u|/2k} \delta\Biggl(\frac{(x_1-x_0)^
{1/k}}{x_2^{1/k}}\Biggr)\bY(Y(u,x_0)v,x_2),
\end{eqnarray*}
giving (\ref{first-supercommutator}). 
\end{proof}

\section{The construction of a weak $(1 \; 2 \; \cdots \; k$)-twisted $V^{\otimes k}$-module structure on a weak $\sigma$-twisted $V$-module $(M_\sigma, Y_\sigma)$ for $k$ even}\label{tensor-product-twisted-construction-section}
\setcounter{equation}{0}

Let $M_\sigma =(M_\sigma,Y_\sigma)$ be a weak $\sigma$-twisted $V$-module. Now we begin our construction of a weak $g$-twisted $V^{\otimes k}$-module structure on $M_\sigma$ when $k$ is an even positive integer and $g = (1 \; 2\; \cdots \; k)$.  Since establishing the properties of $\Delta_k(x)$ as in Section \ref{Delta-section} in the super case and proving the supercommutator (\ref{first-supercommutator}), our construction of a weak $g$-twisted $V^{\otimes k}$-module structure on $M_\sigma$ in the case when $k$ is even follows the same spirit of the construction as in \cite{BDM} and \cite{B-superpermutation-odd}, but now with careful modifications for the fact that we are working with the $\sigma$-twisted vertex operators on the weak $\sigma$-twisted module $M_\sigma$ rather than just untwisted vertex operators on an untwisted weak module as in \cite{BDM} and \cite{B-superpermutation-odd}.

That is, for $k$ even, we construct these weak $g$-twisted $V^{\otimes k}$-modules
by first defining $g$-twisted vertex operators on a weak parity-twisted
$V$-module $M_\sigma$ for a set of generators which are mutually local (see \cite{Li2}).  These $g$-twisted vertex operators generate a local system which is a vertex superalgebra.  We then construct a homomorphism of vertex superalgebras {}from $V^{\otimes k}$ to this local system which thus gives a weak $g$-twisted $V^{\otimes k}$-module structure on $M_\sigma$.

For $u\in V$, and $j = 0, \dots, k-1$, set
\begin{equation}  \label{define-g-twist} 
Y_g(u^1,x) = \bY(u,x)\quad \mbox{and}   \quad
Y_g(u^{j+1},x) = \lim_{x^{1/k}\to \eta^{j} x^{1/k}}  Y_g(u^1,x)  . 
\end{equation}

\begin{rem}\label{generalized-remark}{\em
From the supercommutator (\ref{first-supercommutator}) for $\bar{Y}_\sigma$, we see that defining $g$-twisted operators as above for the case when $k$ is odd, can not result in a twisted module structure on $M_\sigma$ due to appearance of the extra term involving $(x_2^{-1} (x_1 - x_0))^{|u|/2k}$. This parallels the obstruction as discussed in \cite{B-superpermutation-odd} in trying to put a $g$-twisted $V^{\otimes k}$-module structure on an untwisted weak $V$-module if $k$ is even.  That is, we get the same obstruction term in that case---see \cite{B-superpermutation-odd}, Lemma 4.3 and Remarks 4.1 and 5.1.
}
\end{rem}

Note that $Y_g(u^j,x)=\sum_{p=0}^{k-1}Y_g^p(u^j,x)$
where $Y_g^p(u^j,x)=\sum_{n\in \frac{p}{k} + \Z}u^j_nx^{-n-1}$.

\begin{lem}\label{l3.3} Let $u,v\in V$ of homogeneous sign.  Then we have the supercommutator
\begin{multline}\label{3.1}
[Y_g(u^j,x_1),Y_g(v^m,x_2)] \\
= \; \Res_{x_0}\frac{1}{k}x_2^{-1}\delta
\Biggl(\frac{\eta^{j-m}(x_1-x_0)^{1/k}}{x_2^{1/k}}\Biggr)Y_g((Y(u,x_0)v)^m,x_2) 
\end{multline}
where $(Y(u,x_0)v)^m=\sum_{n\in\Z}(u_nv)^m x_0^{-n-1},$
and
\begin{multline}\label{3.2}
[Y_g^p(u^j,x_1),Y_g(v^m,x_2)] \\
= \Res_{x_0}\frac{1}{k}x_2^{-1}\eta^{(m-j)p}\left(\frac{x_1-x_0}
{x_2}\right)^{-p/k}\delta\left(\frac{x_1-x_0}
{x_2}\right)Y_g((Y(u,x_0)v)^m,x_2).
\end{multline}
\end{lem}

\begin{proof} By Lemma \ref{l3.2}, equation (\ref{3.1}) holds if $j=m =1$ and $k$ is even.
Then using (\ref{define-g-twist}), we obtain equation (\ref{3.1})
for any $j,m = 1,...,k$. Equation (\ref{3.2}) is a direct consequence of
(\ref{3.1}). \end{proof}


By Lemma \ref{l3.3} for $u,v\in V$ of homogeneous sign, there exists a positive integer
$N$ such that 
\begin{equation}\label{a3.3}
(x_1-x_2)^N [Y_g(u^j,x_1), Y_g(v^m,x_2)]=0.
\end{equation}
Taking the limit $x^{1/k} \longrightarrow \eta^{j-1} x^{1/k}$
in Lemma \ref{l3.1}, for $j=1,\dots,k$, we have
\begin{equation}
Y_g(L(-1)u^j,x)=\frac{d}{dx}Y_g(u^j,x).
\end{equation} 
Thus the operators $Y_g(u^j,x)$ for $u \in V$, and for $j=1,\dots,k$, are mutually local and generate a local system $A$ of weak twisted vertex operators on $(M_\sigma , L(-1))$ in the sense of \cite{Li2}. 

Let $\rho$ be a map {}from $A$ to $A$ such that $\rho Y_g(u^j,x)=Y_g(u^{j-1},x)$ for
$u\in V$ and $j=1,...,k$.  By Theorem 3.14 of \cite{Li2}\footnote{There is a
typo in the statement of Theorem 3.14 in \cite{Li2}.  The $V$ in the theorem
should be $A$.  That is, the main result of the theorem is that the
local system $A$ of the theorem has the structure of a vertex
superalgebra.}, the local system $A$ generates a vertex superalgebra we
denote by $(A,Y_A)$, and $\rho$ extends to an automorphism of $A$ of
order $k$ such that $M_\sigma$ is a natural weak generalized $\rho$-twisted $A$-module
in the sense that $Y_A(\alpha(x),x_1)=\alpha(x_1)$ for $\alpha(x)\in A$
are $\rho$-twisted vertex operators on $M_\sigma$.

\begin{rem} {\em $\rho$ is given by 
\begin{equation}
\rho a(x) =  \lim_{x^{1/k} \to \eta^{-1}x^{1/k}}  a(x) 
\end{equation}
for $a(x)\in A$; see \cite{Li2}.}
\end{rem}

Let $A^j=\{c(x)\in A| \rho c(x)=\eta^j c(x)\}$ and $a(x)\in A^j$ of homogeneous sign in $A$.
For any integer $n$ and $b(x)\in A$ of homogeneous sign, the operator $a(x)_{n}b(x)$ is an
element of $A$ given by
\begin{eqnarray}\label{3.3}
a(x)_{n}b(x)={\rm Res}_{x_{1}}{\rm
Res}_{x_{0}}\left(\frac{x_{1}-x_{0}}{x}\right)^{j/k}x_{0}^{n}\cdot X
\end{eqnarray}
where 
\[ X=x_{0}^{-1}\delta\left(\frac{x_{1}-x}{x_{0}}\right)a(x_{1})
b(x)- (-1)^{|a||b|} x_{0}^{-1}
\delta\left(\frac{x-x_{1}}{-x_{0}}\right)b(x)a(x_{1}).\]
Or, equivalently, $a(x)_{n}b(x)$ is defined by:
\begin{eqnarray}\label{3.4}
\sum_{n\in \mathbb{Z}}\left(a(x)_{n}b(x)\right)x_{0}^{-n-1}
={\rm Res}_{x_{1}}\left(\frac{x_{1}-x_{0}}{x}\right)^{j/k}
\cdot X.
\end{eqnarray}
Thus following \cite{Li2}, for $a(z)\in A^j$, we define $Y_A(a(z),x)$ 
by setting $Y_A(a(z), x_0)b(z)$ equal to (\ref{3.4}).
  
\begin{lem}\label{l3.5} For $u, v \in V$ of homogeneous sign, we have the supercommutator
\[ [Y_A(Y_g(u^j,x),x_1), Y_A(Y_g(v^m,x),x_2)]=0\]
for $j,m = 1,\dots, k$, with $j \neq m$.
\end{lem}

\begin{proof} The proof is analogous to the proof of Lemma 3.6 in \cite{BDM} where we use the vertex superalgebra structure of $A$ rather than just the vertex algebra structure and we use the supercommutators of Lemma \ref{l3.3}. \end{proof}

Let $Y_g(u^i,z)_n$ for $n \in \mathbb{Z}$ denote the coefficient of $x^{-n-1}$ in the vertex operator $Y_A(Y_g(u^i,z), x)$ for $u \in V$.  That is
\[Y_A(Y_g(u^i,z), x) = \sum_{n \in \Z} Y_g(u^i,z)_n \; x^{-n-1} \in (\mathrm{End} \; A) [[x, x^{-1}]].\]

\begin{lem}\label{l3.6} For $u_1,...,u_k\in V$, we have  
\begin{multline*}
Y_A(Y_g(u^1_1,z)_{-1}\cdots Y_g(u_{k-1}^{k-1},z)_{-1}Y_g(u_k^k,z),x)  \\
 = Y_A(Y_g(u^1_1,z),x)\cdots Y_A(Y_g(u_{k-1}^{k-1},z),x)Y_A(Y_g(u_k^k,z),x) .
\end{multline*}
\end{lem}

\begin{proof} The proof is analogous to the proof of Lemma 3.7 of \cite{BDM}, and Lemma 5.5 of \cite{B-superpermutation-odd}, but modified to the current setting.  From the Jacobi identity on $A$, and Lemma \ref{l3.5}, we have, for $1\leq i<j\leq k$,
\begin{eqnarray*}
\lefteqn{Y_A(Y_g(u^i, z)_{-1} Y_g(v^j,z), x)}\\
&=& \mathrm{Res}_{x_1} \mathrm{Res}_{x_0} x_0^{-1} \left( x^{-1}_0\delta\left(\frac{x_1-x}{x_0}\right) Y_A(Y_g(u^i,z),x_1)Y_A(Y_g(v^j,z),x) \right. \\
& & \quad \left. - (-1)^{|u||v|}
x^{-1}_0\delta\left(\frac{x-x_1}{-x_0}\right) Y_A(Y_g(v^j,z),x)Y_A(Y_g(u^i,z),x_1)\right)\\
&=& \mathrm{Res}_{x_1} \Bigl(  (x_1-x)^{-1} Y_A(Y_g(u^i,z),x_1)Y_A(Y_g(v^j,z),x)  \\
& & \quad - (-1)^{|u||v|} (x-x_1)^{-1} Y_A(Y_g(v^j,z),x)Y_A(Y_g(u^i,z),x_1)\Bigr)\\
&=& \sum_{n <0} Y_g(u^i,z)_n x^{-n-1} Y_A(Y_g(v^j,z),x) \\
& & \quad - (-1)^{|u||v|}  Y_A(Y_g(v^j,z),x) \sum_{n \geq 0} Y_g(u^i,z)_n x^{-n-1}  \\
&=&  \sum_{n \in \Z} Y_g(u^i,z)_n x^{-n-1} Y_A(Y_g(v^j,z),x) \\
&=& Y_A(Y_g(u^i, z), x) Y_A(Y_g(v^j,z), x).
\end{eqnarray*}
The result follows by induction.  
\end{proof}

Define the map $f : \vtk \longrightarrow A$ by
\begin{eqnarray*}
f: \vtk &\longrightarrow& A\\
u_1\otimes\cdots \otimes u_k = (u_1^1)_{-1}\cdots (u_{k-1}^{k-1})_{-1}u^k_k \! \! \! &\mapsto& \! \! \! 
Y_g(u_1^1,z)_{-1}\cdots 
Y_g(u_{k-1}^{k-1},z)_{-1}Y_g(u_k^k,z) 
\end{eqnarray*}
for $u_1,...,u_k\in V$. Then $f(u^j)=Y_g(u^j,z).$

\begin{lem}\label{l3.7} $f$ is a homomorphism of vertex superalgebras.
\end{lem}

\begin{proof} We follow the spirit of the proof of \cite{B-superpermutation-odd} and \cite{BDM}, but must be careful when we need properties of the $\sigma$-twisted vertex operators $Y_\sigma$ on $M_\sigma$ rather than the less complicated case of needing only vertex operators on a weak $V$-module as in \cite{B-superpermutation-odd} and \cite{BDM}.

We need to show that
\[fY(u_1\otimes\cdots \otimes u_k,x)=Y_A( Y_g(u_1^1,z)_{-1}\cdots 
Y_g(u_{k-1}^{k-1},z)_{-1}Y_g(u_k^k,z),x)f\]
for $u_i\in V.$ 
Take $v_i\in V$ for $i=1,...,k.$ Then 
\begin{eqnarray*}
\lefteqn{fY(u_1\otimes\cdots \otimes u_k,x)(v_1\otimes \cdots\otimes v_k) }\\
&=& \! \! (-1)^s f(Y(u_1,x)v_1\otimes\cdots Y(u_k,x)v_k)\\
&=& \! \! (-1)^s Y_g(Y(u_1^1,x)v_1^1,z)_{-1}\cdots Y_g(Y(u_{k-1}^{k-1},x)v_{k-1}^{k-1},z)_{-1}
Y_g(Y(u_k^k,x)v_k^k,z) 
\end{eqnarray*}
for $s = \sum_{j=1}^{k-1} |v_j| \sum_{i = j + 1}^k |u_i|$.

By Lemma \ref{l3.6}, we have
\begin{multline*} 
Y_A( Y_g(u_1^1,z)_{-1}\cdots Y_g(u_{k-1}^{k-1},z)_{-1}Y_g(u_k^k,z),x)
f(v_1\otimes \cdots \otimes v_k) \\
= Y_A(Y_g(u^1_1,z),x)\cdots Y_A(Y_g(u_{k-1}^{k-1},z),x)Y_A(Y_g(u_k^k,z),x)
Y_g(v_1^1,z)_{-1}\\
\cdots Y_g(v_{k-1}^{k-1},z)_{-1}Y_g(v_k^k,z).
\end{multline*}
By Lemma \ref{l3.5}, it is enough to show that 
$$ Y_g(Y(u^j,x)v^j,z)=Y_A(Y_g(u^j,z),x)Y_g(v^j,z)$$
for $u,v\in V$ and $j=1,...,k.$ In fact, in view of the
relation between $Y(u^1,z)$ and $Y(u^j,z)$ for 
$u\in V,$  we only need to prove the case $j=1.$

By Proposition \ref{psun1},
\begin{eqnarray*}
Y_g(Y(u^1,x_0)v^1,x_2) &=& Y_\sigma(\Delta_k(x_2)Y(u,x_0)v,x_2^{1/k})\\
&=& Y_\sigma(Y(\Delta_k(x_2+x_0)u,(x_2+x_0)^{1/k}-x_2^{1/k})\Delta_k(x_2)v,x_2^{1/k}).
\end{eqnarray*}
On the other hand,
\[Y_A(Y_g(u^1,x_2),x_0)Y_g(v^1,x_2)=\sum_{p=0}^{k-1}\Res_{x_1}
\left(\frac{x_1-x_{0}}{x_2}\right)^{p/k}X\]
where
\begin{multline}
X=x_{0}^{-1}\delta\left(\frac{x_1-x_2}{x_{0}}\right)Y_g(u^1,x_1)
Y_g(v^1,x_2)\\
-(-1)^{|u||v|}x_{0}^{-1}\delta\left(\frac{x_2-x_1}{-x_{0}}\right)
Y_g(v^1,x_2)Y_g(u^1,x_1).
\end{multline}

By equation (\ref{a3.3}), there exists a positive integer $N$ such that
\[(x_1-x_2)^NY_g(u^1,x_1)Y_g(v^1,x_2)= (-1)^{|u||v|} (x_1-x_2)^NY_g(v^1,x_2)Y_g(u^1,x_1).\]
Thus, using the three-term $\delta$-function identity (\ref{three-term-delta}), we have 
\begin{eqnarray*}
X &=& x_{0}^{-1}\delta\left(\frac{x_1-x_2}{x_{0}}\right)Y_g(u^1,x_1)
Y_g(v^1,x_2)\\
& &\quad - (-1)^{|u||v|} x_{0}^{-1} \delta\left(\frac{x_2-x_1}{-x_{0}}\right)
x_0^{-N}(x_1-x_2)^NY_g(v^1,x_2)Y_g(u^1,x_1)\\
&=& x_{0}^{-1}\delta\left(\frac{x_1-x_2}{x_{0}}\right)x_0^{-N}
\left((x_1-x_2)^NY_g(u^1,x_1)Y_g(v^1,x_2)\right)\\
& & \quad - (-1)^{|u||v|} x_{0}^{-1} \delta\left(\frac{x_2-x_1}{-x_{0}}\right)
x_0^{-N}\left((x_1-x_2)^NY_g(u^1,x_1)Y_g(v^1,x_2)\right)\\
&=& x_2^{-1}x_0^{-N}\delta\left(\frac{x_1-x_0}{x_2}\right)
\left((x_1-x_2)^NY_g(u^1,x_1)Y_g(v^1,x_2)\right) .
\end{eqnarray*}

Therefore using the $\delta$-function relation (\ref{delta-function2}), 
we have
\begin{multline*}
Y_A(Y_g(u^1,x_2),x_0)Y_g(v^1,x_2) \\
= \Res_{x_1} x_0^{-N} x_2^{-1} \delta\Biggl(\frac{(x_1-x_0)^{1/k}}{x_2^{1/k}}\Biggr) \left((x_1-x_2)^NY_g(u^1,x_1)Y_g(v^1,x_2)\right).
\end{multline*}

Let $x$ be a new formal variable which commutes with $x_0,x_1,x_2.$ 
Then using the  $\delta$-function identities of Section \ref{formal-calculus-section} and the definition of $Y_g$ given by (\ref{define-g-twist}), we have 
\begin{eqnarray*}
\lefteqn{x_2^{-1/k}\delta\Biggl(\frac{x_1^{1/k}-x}{x_2^{1/k}}
\Biggr)\left((x_1-x_2)^N Y_g(u^1,x_1)Y_g(v^1,x_2)\right) }\\
&=&  x^{-1}\delta\Biggl(\frac{x_1^{1/k}-x_2^{1/k}}{x}\Biggr)
\left((x_1-x_2)^N Y_g(u^1,x_1)Y_g(v^1,x_2)\right)\\
& & \quad  - \; x^{-1} \delta\Biggl(\frac{-x_2^{1/k}+x_1^{1/k}}{x}
\Biggr)\left((x_1-x_2)^N Y_g(u^1,x_1)Y_g(v^1,x_2)\right)\\
&=&  (x_1-x_2)^N x^{-1}\delta\Biggl(\frac{x_1^{1/k}-x_2^{1/k}}{x}\Biggr)
Y_\sigma (\D(x_1)u,x_1^{1/k})Y_\sigma (\D(x_2)v,x_2^{1/k})\\
& &\quad  -\; (x_1-x_2)^N x^{-1} \delta\Biggl(\frac{-x_2^{1/k} +
x_1^{1/k}}{x}\Biggr) Y_\sigma (\D(x_2)v,x_2^{1/k})Y_\sigma (\D(x_1)u,x_1^{1/k})\\
&=&  (x_1-x_2)^N x_2^{-1/k}\delta\Biggl(\frac{x_1^{1/k}-x}{x_2^{1/k}}\Biggr)
Y_\sigma (Y(\D(x_1)u,x)\D(x_2)v,x_2^{1/k}).
\end{eqnarray*}
Note that the first term in the above formula is well defined when
$x$ is replaced by  $x_1^{1/k}-(x_1-x_0)^{1/k}$, and therefore the last 
term is also well defined under this substitution. Thus since $k$ is even
\begin{eqnarray*}
\lefteqn{ x_0^{-N}x_2^{-1/k}\delta\Biggl(
\frac{(x_1-x_0)^{1/k}}{x_2^{1/k}}\Biggr)\left((x_1-x_2)^N
Y_g(u^1,x_1)Y_g(v^1,x_2)\right) }\\
&=&  x_2^{-1/k}\delta\Biggl(\frac{(x_1-x_0)^{1/k}}{x_2^{1/k}}\Biggr)
Y_\sigma (Y(\D(x_1)u,x_1^{1/k}-(x_1-x_0)^{1/k})\D(x_2)v,x_2^{1/k})\\
&=& x_2^{-1/k}\delta\Biggl(\frac{(x_1-x_0)^{1/k}}{x_2^{1/k}}\Biggr)\\
& & \quad 
Y_\sigma (Y(\D(x_2+x_0)u,(x_2+x_0)^{1/k}-x_2^{1/k})\D(x_2)v,x_2^{1/k}).
\end{eqnarray*}

Finally using Proposition \ref{psun1}, we have
\begin{eqnarray*}
\lefteqn{Y_A(Y_g(u^1,x_2),x_0)Y_g(v^1,x_2) }\\
&=&\Res_{x_1}x_2^{-1}\delta\Biggl(\frac{(x_1-x_0)^{1/k}}{x_2^{1/k}}\Biggr)\\
& & \quad Y_\sigma (Y(\D(x_2+x_0)u,(x_2+x_0)^{1/k}-x_2^{1/k}) \D(x_2)v,x_2^{1/k})\\
&=&  Y_\sigma (Y(\D(x_2+x_0)u,(x_2+x_0)^{1/k}-x_2^{1/k})\D(x_2)v,x_2^{1/k})\\
&=& Y_g(Y(u^1,x_0)v^1,x_2),
\end{eqnarray*}
as desired.  \end{proof}

Let $(M_\sigma,Y_\sigma)$ be a weak $\sigma$-twisted $V$-module, $k$ a positive even integer, and $g = (1 \; 2 \; \cdots \; k)$.  Define $T_g^k(M_\sigma,Y_\sigma) = (T_g^k(M_\sigma),
Y_g) = (M_\sigma, Y_g)$.  That is $T_g^k(M_\sigma, Y_\sigma)$ is $M_\sigma$ as the underlying vector space and the vertex operator $Y_g$ is given by (\ref{define-g-twist}).

Now we state our first main theorem of the paper.
\begin{thm}\label{main1} 
$(T_g^k(M_\sigma),Y_g)$ is a weak $g$-twisted $V^{\otimes k}$-module such
that $T_g^k(M_\sigma)=M_\sigma$, and $Y_g$, defined by (\ref{define-g-twist}), is the
linear map {}from $V^{\otimes k}$ to $(\End \; T_g^k(M_\sigma))[[x^{1/k}, \\
x^{-1/k}]]$
defining the twisted module structure. Moreover, 

(1) $(M_\sigma, Y_\sigma)$ is an irreducible weak $\sigma$-twisted $V$-module if and only if
$(T_g^k(M_\sigma), Y_g)$ is an irreducible weak $g$-twisted $V^{\otimes
k}$-module.

(2) $M_\sigma$ is a weak admissible $\sigma$-twisted $V$-module if and only if $T_g^k(M_\sigma)$ is a weak admissible $g$-twisted $V^{\otimes k}$-module.

(3) $M_\sigma$ is an ordinary $\sigma$-twisted $V$-module if and only if $T_g^k(M_\sigma)$ is an ordinary $g$-twisted $V^{\otimes k}$-module.
\end{thm}

\begin{proof} It is immediate {}from Lemma \ref{l3.7} that $T_g^k(M_\sigma)=M_\sigma$ is a weak $g$-twisted $V^{\otimes k}$-module with $Y_g(u^1,x)=\bar Y_\sigma(u,x)$.  Note
that with 
\begin{equation}\label{Delta-inverse}
\Delta_k (x)^{-1} = ( x^{1/2k})^{-( 1 - k) 2L(0)}  (k^{1/2})^{2L(0)} \exp 
\Biggl( -\sum_{j \in \ZZ} a_j x^{- j/k} L(j) \Biggr),
\end{equation}
we have 
\begin{equation}\label{for-grading}
Y_g((\Delta_k(x)^{-1}u)^1,x)=\bar Y_\sigma (\Delta_k(x)^{-1}u,x)=Y_\sigma(u,x^{1/k}) ,
\end{equation}
and all twisted vertex operators $Y_g(v,x)$ for $v\in \vtk$ can
be generated {}from $Y_g(u^1,x)$ for $u\in V.$  

It is clear now that $M_\sigma$ is an irreducible weak $\sigma$-twisted $V$-module if and only if $T_g^k(M_\sigma)$ is an irreducible weak $g$-twisted $V^{\otimes k}$-module, proving statement (1).

For statement (2), we first assume that $M_\sigma$ is a weak admissible $\sigma$-twisted $V$-module.  Then from the definition of a weak admissible $\sigma$-module and Remark \ref{parity-grading-remark}, we have $M_\sigma = \coprod_{n\in \mathbb{N}} M_\sigma(n)$ such that for $m \in \frac{1}{2} \mathbb{Z}$, the component operator $u_m^\sigma$ of $Y_\sigma(u, z)$, satisfies $u_m^\sigma M_\sigma(n)\subset M_\sigma(\wt \; u-m-1+n)$ if $u\in V$ is of homogeneous weight.  

Define a $\frac{1}{k}\N$-gradation on $T_g^k(M_\sigma)$ such that $T_g^k(M_\sigma)(n/k) =
M_\sigma(n)$ for $n\in \N$.  
Recall that $Y_g(v,x) = \sum_{m \in \frac{1}{k} \mathbb{Z}} v^g_m x^{-m-1}$ for $v\in \vtk$.  We need to show that $v^g_mT_g^k(M_\sigma)(n)\subset T_g^k(M_\sigma)(\wt \; v-m-1+n)$ for $m\in\frac{1}{k} \Z$, and $n \in \N$.  Since all twisted vertex operators $Y_g(v,x)$ for $v\in \vtk$ can be generated {}from $Y_g(u^1,x)$ for $u\in V$, it is
enough to show $(u^1)^g_m T_g^k(M_\sigma)(n) \subset T_g^k(M_\sigma)(\wt \; u-m-1+n)$.

Let $u\in V_p$ for $p\in\frac{1}{2} \Z$.  Then 
\[\Delta_k(x)u=\sum_{j=0}^{\infty}u(j)x^{p/k - p - j/k}\]
where $u(j)\in V_{p-j}.$ Thus
\begin{eqnarray*}
Y_g(u^1,x) &= &Y_\sigma (\Delta_k(x)u,x^{1/k})\ = \ \sum_{j=0}^{\infty}Y_\sigma(u(j),x^{1/k}) x^{p/k - p - j/k} \\
&=& \sum_{j=0}^{\infty}\sum_{l \in \frac{1}{2} \mathbb{Z}} u(j)^\sigma_l x^{(-l-1)/k} x^{p/k - p - j/k} \nonumber.
\end{eqnarray*}
Therefore for $m \in \frac{1}{k} \mathbb{Z}$, we have
\[(u^1)^g_m=\sum_{j=0}^{\infty}u(j)_{(1-k)p-j-1+km+k}.\]
Since the weight of $u(j)_{(1-k)p-j-1+km+k}$ is $k(p-m-1)$, we see that for $n \in \frac{1}{k} \N$, we have $(u^1)^g_mT_g^k(M_\sigma)(n)=(u^1)^g_mM_\sigma(kn) \subset M_\sigma (k(p-m-1+n)) =T_g^k(M_\sigma)(p-m-1+n)$, showing that $T_g^k(M_\sigma)$ is a weak admissible $g$-twisted $V^{\otimes k}$-module.

Conversely, we assume that $T_g^k(M_\sigma)$ is a weak admissible $g$-twisted $V^{\otimes k}$-module, i.e., we have $T_g^k(M_\sigma) = \coprod_{n\in \frac{1}{k}\mathbb{N}} T_g^k(M_\sigma)(n)$ such that for $m \in \frac{1}{k} \mathbb{Z}$, the component operator $u^g_m$ of $Y_g(u, x)$ satisfies $u^g_m T_g^k(M_\sigma)(n)\subset T_g^k(M_\sigma)(\wt \; u-m-1+n)$ if $u\in V^{\otimes k}$ is of homogeneous weight.   Define an $\N$-gradation on $M_\sigma$ such that $M_\sigma (n) = T_g^k(M_\sigma)(n/k)$ for $n\in \N.$ 

Note that by again letting $u\in V_p$ for $p\in\frac{1}{2} \Z$, then 
\[\Delta_k(x)^{-1} u=\sum_{j=0}^{\infty}u[j]x^{ p -p/k - j}\]
where $u[j]\in V_{p-j}.$  Thus Equation (\ref{for-grading}) implies 
\begin{eqnarray*}
Y_\sigma (u, x) &=& Y_g((\Delta_k(x^k)^{-1}u)^1,x^k) \ = \ \sum_{j=0}^\infty Y_g(u[j]^1, x^k) x^{pk - p -jk} \\
&=& \sum_{j=0}^\infty \sum_{l \in \frac{1}{k} \mathbb{Z}} (u[j]^1)^g_l x^{-kl-k} x^{pk - p -jk} 
\end{eqnarray*}
and thus for $m \in \frac{1}{2} \mathbb{Z}$
\[u_m^\sigma = \sum_{j=0}^\infty (u[j]^1)^g_{\frac{1}{k}((k-1)p -jk -k + m + 1) }.\]
The weight of $(u[j]^1)^g_{\frac{1}{k}((k-1)p -jk -k + m + 1) }$ is $\frac{1}{k}(p-m-1)$.  Therefore for the weak $\sigma$-twisted $V$-module $M_\sigma$, we have for $n \in \N$, that $u^\sigma_m M_\sigma (n) = u^\sigma_m T_g^k(M_\sigma) (n/k) \subset T_g^k(M_\sigma) ( \frac{1}{k} (p - m - 1 + n)) = M_\sigma(p-m-1 + n)$, finishing the proof of (2).

In order to prove (3) we write $Y_g(\bar\omega,x) = \sum_{n\in\Z}
L^g(n)x^{-n-2}$ where $\bar\omega=\sum_{j=1}^k\omega^j$.  We have
\[Y_g(\bar\omega,x)=\sum_{j=0}^{k-1} \ \lim_{x^{1/k}\mapsto \eta^{-j}x^{1/k}}
  Y_g(\omega^1,x).\]
It follows {}from (\ref{sun1}) that 
\begin{equation}\label{L(0)-conversion}
L^g(0)=\frac{1}{k}L^\sigma (0)+\frac{(k^2-1)c}{24k}.  
\end{equation}
This immediately implies (3).
\end{proof}

Let $V$ be an arbitrary vertex operator superalgebra and $g$ an automorphism of $V$ of finite order. We denote the categories of weak, weak admissible and ordinary generalized $g$-twisted $V$-modules by $\mathcal{ C}^g_w(V),$ $\mathcal{ C}^g_a(V)$ and $\mathcal{ C}^g(V)$, respectively.  

Now again consider the vertex operator superalgebra $V^{\otimes k}$ and the
$k$-cycle $g = (1 \; 2 \; \cdots \; k)$.  Define
\begin{eqnarray*}
T_g^k: \mathcal{ C}^\sigma_w(V) &\longrightarrow& \mathcal{ C}^g_w(V^{\otimes k})\\
  (M_\sigma,Y_\sigma) &\mapsto& (T_g^k(M_\sigma),Y_g) = (M_\sigma,Y_g)\\
     f  &\mapsto& T_g^k(f) = f
\end{eqnarray*}
for $(M_\sigma,Y_\sigma)$ an object and $f$ a morphism in $\mathcal{ C}^\sigma_w(V)$.  

The following corollary to Theorem \ref{main1} follows immediately.

\begin{cor}\label{c3.10} 
If $k$ is even, then $T_g^k$ is a functor {}from the category $\mathcal{ C}^\sigma_w(V)$ of weak parity-twisted $V$-modules to the category $\mathcal{ C}^g_w(\vtk)$ of weak $g = (1 \; 2 \; \cdots \; k)$-twisted $V^{\otimes k}$-modules, such that: (1) $T_g^k$ preserves irreducible
objects; (2) The restrictions of $T_g^k$ to $\mathcal{ C}^\sigma_a(V)$ and $\mathcal{
C}^\sigma(V)$ are functors {}from $\mathcal{ C}^\sigma_a(V)$ and $\mathcal{ C}^\sigma(V)$ to $\mathcal{C}^g_a(\vtk)$ and $\mathcal{ C}^g(\vtk)$, respectively.
\end{cor}

In the next section we will construct a functor $U_g^k$, in the case when $k$ is even, {}from the
category $\mathcal{ C}^g_w(\vtk)$ to the category $\mathcal{ C}^\sigma_w(V)$ such
that $U_g^k \circ T_g^k = id_{\mathcal{ C}^\sigma_w(V)}$ and $T_g^k \circ U_g^k = 
id_{\mathcal{ C}^g_w(\vtk)}$.

\section{Constructing a weak $\sigma$-twisted $V$-module structure on a 
weak $g = (1 \; 2 \; \cdots \; k)$-twisted $V^{\otimes k}$-module for $k$ even}\label{classification-section}
\setcounter{equation}{0}

For $k \in \ZZ$ and $g = (1\; 2\; \cdots \; k)$, let $M_g=(M_g,Y_g)$ be a 
weak $g$-twisted $\vtk$-module.  Motivated by the construction 
of weak $g$-twisted $\vtk$-modules {}from weak $\sigma$-twisted $V$-modules in 
Section \ref{tensor-product-twisted-construction-section}, we consider 
\begin{equation}\label{define U}
Y_g((\Delta_k(x^k)^{-1}u)^1,x^k)
\end{equation}
for $u\in V$ where $\Delta_k (x)^{-1}$ is given by 
(\ref{Delta-inverse}).  Note that (\ref{define U}) is multivalued since 
$Y_g((\Delta_k(x)^{-1}u)^1,x) \in (\mathrm{End} \, M_g) [[x^{1/2k}, x^{-1/2k}]]$.   
Thus we define 
\begin{equation}\label{define-sigma-twist}
Y_\sigma(u,x) = Y_g((\Delta_k(x^k)^{-1}u)^1,x^k)
\end{equation} 
to be the unique formal Laurent series in $(\mathrm{End} \, M_g) [[x^{1/2}, x^{-1/2}]]$ given by  
taking $(x^{k})^{1/2k} = x^{1/2}$.

Our goal in this section is to construct a functor $U_g^k : \mathcal{ C}_w^g(\vtk) \rightarrow \mathcal{ C}^\sigma_w(V)$ with $U_g^k(M_g,Y_g) = (U_g^k(M_g),Y_\sigma) = (M_g,Y_\sigma)$ for the case when $k$ is even.  If we instead define $Y_\sigma$ by taking $(x^{2k})^{1/k} = \eta^j x^{1/2}$ for $\eta$ a fixed primitive $k$-th root of unity and  $j=1,\dots,k-1$, then $(M_g,Y_\sigma)$ will not be a weak $\sigma$-twisted $V$-module.  Further note that this implies that if we allow $x=z$ to be a complex number and if we define $z^{1/k}$ using the principal branch of the logarithm, then much of our work in this section is valid if and only if $-\pi/k < 
\mathrm{arg} \; z < \pi/k$.

\begin{lem}\label{l4.1} For $u\in V,$ we have
\begin{eqnarray*}
Y_\sigma(L(-1)u,x) &=& \left(\frac{d}{dx}((x^k)^{1/k})\right)\frac{d}{dx}Y_\sigma(u,x)\\
&=& \frac{d}{dx} Y_\sigma(u,x)
\end{eqnarray*}
on $M_g$.  Thus the $L(-1)$-derivative property holds for $Y_\sigma$ on $M_g$.
\end{lem}

\begin{proof} The proof is similar to that of Lemma \ref{l3.1}.
By Lemma \ref{c2.5} we have
\[\D(x)^{-1}L(-1)-kx^{-1/k + 1}L(-1)\D(x)^{-1} = kx^{-1/k 
+ 1}\frac{d}{dx}\D(x)^{-1}.\]
Making the change of variable $x\to x^k$ gives 
\[\D(x^k)^{-1} L(-1) - k(x^k)^{-1/k} x^k L(-1) \D(x^k)^{-1} =
(x^k)^{-1/k}x\frac{d}{dx}\D(x^k)^{-1}.\]
Thus if $(x^k)^{1/k} = \eta^j x$, we have
\begin{eqnarray*}
\lefteqn{\frac{d}{dx}Y_g((\D(x^k)^{-1}u)^1,x^k)}  \\
&=& Y_g((\frac{d}{dx}\D(x^k)^{-1}u)^1,z^k)+\frac{d}{dx}
Y_g((\D(z^k)^{-1}u)^1,x^k)|_{x=z}\\
&=& Y_g((\frac{d}{dx}\D(x^k)^{-1}u)^1,x^k)+
kx^{k-1}Y_g(L(-1)(\D(x^k)^{-1}u)^1,x^k)\\
&=& \eta^j Y_g((\D(x^k)^{-1}L(-1)u)^1,x^k).
\end{eqnarray*}
Since by definition $Y_\sigma(u,x) = Y_g((\D(x^k)^{-1}u)^1,x^k)$ with
$(x^k)^{1/k} = x$, the result follows. \end{proof}

\begin{lem}\label{l4.2} Let $u,v\in V$.  Then on $M_g$, we have the supercommutator
\begin{multline}
[Y_\sigma (u,x_1),Y_\sigma (v,x_2)] \\
=
\Res_{x_0} x_2^{-1}\delta\left(\frac{x_1-x_0}{x_2}\right) \left(\frac{x_1-x_0}{x_2}\right)^{\frac{1}{2}(1-k)|u|} 
Y_\sigma (Y(u, x_0)v,x_2),
\end{multline}
i.e.,
\begin{multline}
[Y_\sigma (u,x_1), Y_\sigma (v,x_2) ] \\
= \left\{ \begin{array}{ll}
\Res_{x_0} x_2^{-1}\delta\left(\frac{x_1-x_0}{x_2}\right) \left(\frac{x_1-x_0}{x_2}\right)^{\frac{|u|}{2}} Y_\sigma (Y(u, x_0)v,x_2) & \mbox{if $k$ is even}\\
\Res_{x_0} x_2^{-1}\delta\left(\frac{x_1-x_0}{x_2}\right) Y_\sigma (Y(u, x_0)v,x_2) & \mbox{if $k$ is odd}
\end{array} \right. .
\end{multline}
\end{lem}

\begin{proof} The proof is similar to the proof of Lemma \ref{l3.2} and is analogous to the proof of Lemma 4.2 in \cite{BDM}, but with the significant change that we go from $g$-twisted operators to $\sigma$-twisted operators, rather than from $g$-twisted operators to untwisted operators.

From the 
twisted Jacobi identity on $(M_g, Y_g)$, we have
\begin{equation}\label{commutator for Lemma 4.2}
 [Y_g(u^1,x_1),Y_g(v^1,x_2)] \; = \; \Res_{x_0}\frac{1}{k}x_2^{-1}
\delta\Biggl(\frac{(x_1-x_0)^{1/k}}{x_2^{1/k}}\Biggr)Y_g(
Y(u^1,x_0)v^1,x_2).
\end{equation}
Therefore, 
\begin{eqnarray*}
\lefteqn{[Y_\sigma(u,x_1),Y_\sigma(v,x_2)]} \\
&=& [Y_g((\D(x^k_1)^{-1}u)^1,x_1^k),Y_g((\D(x_2^k)^{-1}v)^1,x_2^k)] \\
&=& \Res_{x}\frac{1}{k} x_2^{-k} \delta \left( 
\frac{(x_1^k - x)^{1/k}}{x_2} \right)  Y_g( Y( 
(\D(x_1^k)^{-1}u)^1,x) (\D(x_2^k)^{-1}v)^1,x_2^k).
\end{eqnarray*}

We want to make the change of variable $x=x_1^k-(x_1-x_0)^{k}$ where we
choose $x_0$ such that $((x_1 - x_0)^k)^{1/k} = x_1 - x_0$. Then noting
that $(x_1^k - x)^{n/k} |_{x = x_1^k-(x_1-x_0)^{k}} = (x_1 - x_0)^n$ for 
all $n \in \mathbb{Z}$, and using (\ref{residue change of variables}), 
we have
\begin{eqnarray*}
\lefteqn{[Y_\sigma(u,x_1),Y_\sigma(v,x_2)]} \\
&=& \Res_{x_0} x_2^{-k} (x_1-x_0)^{k-1} \delta \left( \frac{
x_1-x_0}{x_2} \right) Y_g( Y((\D(x_1^k)^{-1}u)^1,x_1^k -
(x_1-x_0)^{k})\\
& & \quad (\D(x_2^k)^{-1}v)^1,x_2^k)\\
&=& \Res_{x_0}  x_2^{-1}\delta\left(\frac{x_1-x_0}{x_2}\right) 
Y_g( Y( (\D(x_1^k)^{-1}u)^1, x_1^k - (x_1-x_0)^{k})\\
& & \quad (\D(x_2^k)^{-1}v)^1,x_2^k)\\
&=& \Res_{x_0}  x_1^{-1}\delta\left(\frac{x_2+x_0}{x_1}\right) 
Y_g( Y( (\D(x_1^k)^{-1}u)^1, x_1^k - (x_1-x_0)^{k})\\
& & \quad (\D(x_2^k)^{-1}v)^1,x_2^k)\\
&=& \Res_{x_0} x_1^{-1}\delta\left(\frac{x_2+x_0}{x_1}\right) \left(\frac{x_2+x_0}{x_1}\right)^{\frac{1}{2}(1-k)|u|} 
Y_g((Y(\D((x_2+x_0)^k)^{-1}u, \\
& & \quad (x_2+x_0)^k-x_2^{k})  \D(x_2^k)^{-1}v)^1,x_2^k)\\
&=& \Res_{x_0}  x_2^{-1} \delta \left( \frac{x_1-x_0}{x_2}\right) \left(\frac{x_1-x_0}{x_2}\right)^{\frac{1}{2}(k-1)|u|} 
Y_g((Y(\D((x_2+x_0)^k)^{-1}u, \\
& & \quad (x_2+x_0)^k-x_2^{k}) \D(x_2^k)^{-1}v)^1,x_2^k).
\end{eqnarray*}

Thus the proof is reduced to proving
\begin{eqnarray*} 
Y(\D((x_2+x_0)^k)^{-1}u,(x_2+x_0)^k - x_2^{k})\D(x_2^k)^{-1} = \D(x_2^k)^{-1} Y
\left(u, x_0 \right),
\end{eqnarray*} 
i.e., proving
\begin{equation}\label{3.10}
\D(x_2^k)Y(\D((x_2+x_0)^k)^{-1}u,(x_2+x_0)^k-
x_2^{k})\D(x_2^k)^{-1} =  Y \left(u,  x_0 \right).
\end{equation}
In Proposition \ref{psun1}, substituting $u$, $z$ and $z_0$ with
$ \D((x_2+x_0)^k)^{-1}u,$ $x_2^k$ and $(x_2+x_0)^k
- x_2^{k}$, respectively, gives equation (\ref{3.10}).
\end{proof}

Let $(M_g,Y_g)$ be a weak $g$-twisted $V$-module, for $k$ a positive even integer, and $g = (1 \; 2 \; \cdots \; k)$.  Define $U_g^k(M_g,Y_g) = (U_g^k(M_g),
Y_\sigma) = (M_g, Y_\sigma)$.  That is $U_g^k(M_g, Y_g)$ is $M_g$ as the underlying vector space and the $\sigma$-twisted vertex operator $Y_\sigma$ is given by (\ref{define-sigma-twist}).

\begin{thm}\label{t4.l} Given a  weak $g$-twisted $V$-module $(M_g,Y_g)$, with the notations as above, $U_g^k(M_g,Y_g) = 
(U_g^k(M_g),Y_\sigma) = (M_g,Y_\sigma)$ is a weak $\sigma$-twisted $V$-module.
\end{thm}

\begin{proof} Since the $L(-1)$-derivation property has been proved for $Y_\sigma$ in 
Lemma \ref{l4.1}, we only need to prove the twisted Jacobi identity.   In this setting, the twisted Jacobi identity is equivalent to the supercommutator formula given by Lemma  \ref{l4.2} for the case when $k$ is even, and the associator  formula.  The associator formula states that for $u,v\in V$ and $w \in U_g^k(M_g)$, there exists  a positive integer $n$ such that
\[(x_0+x_2)^{|u|/2 + n}Y_\sigma(u,x_0+x_2)Y_\sigma(v,x_2)w=(x_2+x_0)^{|u|/2 + n}Y_\sigma(Y(u,x_0)v,x_2)w .\]
Here we are using the fact that the eigenspaces for the parity automorphism $\sigma$ are given by $V^0 = V^{(0)}$ and $V^1 = V^{(1)}$.  


Write $u^1=\sum_{j=0}^{k-1}u^1_{(j)}$ where $gu^1_{(j)} =
\eta^ju^1_{(j)}$.  Then {}from the twisted Jacobi identity, we have the
following associator: There exists a positive integer $m$ such that
for $n\geq m,$
\[(x_0+x_2)^{j/k+n}Y_g(u^1_{(j)},x_0+x_2)Y_g(v^1,x_2)w=
(x_2+x_0)^{j/ k+n}Y_g(Y(u^1_{(j)},x_0)v^1,x_2)w\]
for $j=0,...,k-1$.  Replacing $x_2$ by $x_2^k$ and $x_0$ by $(x_0+x_2)^k-x_2^k$ gives 
\begin{multline*}
(x_0+x_2)^{j+kn}Y_g(u^1_{(j)},(x_0+x_2)^k)Y_g(v^1,x_2^k)w\\
= (x_2+x_0)^{j+kn}Y_g(Y(u^1_{(j)},(x_2+x_0)^k-x_2^k)v^1,x_2^k)w.
\end{multline*}

Note that if $a\in \vtk$ such that $ga=\eta^ja$, then $Y_g(a,x)=\sum_{l\in
j/k+\Z}a_nx^{-l-1}$. Thus there exists a positive integer $m_j$
such that if $n_j\geq m_j$, then 
\begin{multline*}
(x_0+x_2)^{n_j}Y_g(u^1_{(j)},(x_0+x_2)^k)Y_g(v^1,x_2^k)w\\
=(x_2+x_0)^{n_j} Y_g(Y(u^1_{(j)},(x_2+x_0)^k-x_2^k)v^1,x_2^k)w
\end{multline*}
for $j=0,...,k-1$.  As a result we see that there exists a positive integer 
$m$ such that if $n\geq m$, then 
\begin{multline*}
(x_0+x_2)^{n}Y_g(u^1,(x_0+x_2)^k)Y_g(v^1,x_2^k)w\\
= (x_2+x_0)^{n}Y_g(Y(u^1,(x_2+x_0)^k-x_2^k)v^1,x_2^k)w.
\end{multline*}

Note that $\Delta_k(x)^{-1} u \in (x^{\frac{1}{2k}})^{(k-1)2 \mathrm{wt} \, u} V[x^{-1}]$.  
Thus for $k$ even, we have that $\Delta_k(x^k)^{-1} u \in x^{-|u|/2} V[x, x^{-1}]$.  Therefore we can write $\Delta_k((x_0+x_2)^k)^{-1}u = (x_0 + x_2)^{-|u|/2}\sum_{j \in \mathbb{N}} u_j(x_0+x_2)^{s_j}$ for some $u_j\in V$ and integers $s_j \in \mathbb{Z}$, and note that this is a finite sum.  Similarly we have a finite sum
$\Delta_k(x_2^k)^{-1}v=x_2^{-|v|/2} \sum_{j\in \mathbb{N}}v_j x_2^{t_j}$ for some $v_j \in
V$ and $t_j \in \mathbb{Z}$.  Thus there exists a positive integer
$m$ such that if $n\geq m$, then
\begin{multline*}
(x_0+x_2)^{n+s_i}Y_g(u_i^1,(x_0+x_2)^k)Y_g(v_j^1,x_2^k)w\\
=(x_2+x_0)^{n+s_i}Y_g(Y(u_i^1,(x_2+x_0)^k-x_2^k)v_j^1,x_2^k)w
\end{multline*}
for all $i,j\in \mathbb{N}$.  Finally, using equation (\ref{3.10}), 
we have for $n\geq m,$ 
\begin{eqnarray*}
\lefteqn{(x_0+x_2)^{|u|/2 + n} Y_\sigma(u,x_0+x_2)Y_\sigma(v,x_2)w} \\
&=& (x_0+x_2)^{|u|/2 + n} Y_g((\D((x_0+x_2)^k)^{-1}u)^1,
(x_0+x_2)^k)\\
& & \quad Y_g((\D(x_2^k)^{-1}v)^1,x_2^k) w\\
&=&  \sum_{i,j \geq 0}(x_0+x_2)^{n+s_i}x_2^{-|v|/2 + t_j}Y_g(u_i^1,(x_0+x_2)^k)
Y_g(v_j^1,x_2^k)w\\
&=&  \sum_{i,j \geq 0}(x_2+x_0)^{n+s_i}x_2^{-|v|/2 + t_j}Y_g(Y(u_i^1,
(x_2+x_0)^k-x_2^k)v_j^1,x_2^k)w\\
&=& (x_2+x_0)^{ |u|/2 + n} Y_g(Y(\D((x_2+x_0)^k)^{-1}u)^{1},(x_2+x_0)^k-x_2^k)\\
& & \quad 
(\D(x_2^k)^{-1}v)^1,x_2^k)w\\
&=& (x_2+x_0)^{|u|/2 + n}  Y_g((\D(x_2^k)^{-1}Y(u,x_0)v)^1,x_2^k)w\\
&=& (x_2+x_0)^{|u|/2 + n}  Y_\sigma(Y(u,x_0)v,x_2)w
\end{eqnarray*}
completing the proof.
\end{proof}

\begin{thm}\label{t4.ll} 
For $k$ an even positive integer, and $g = (1 \; 2\; \cdots \; k)$, the map
$U_g^k$ is a functor {}from the category $\mathcal{ C}_w^g(\vtk)$ of
weak $g$-twisted $\vtk$-modules to the category $\mathcal{ C}_w^\sigma(V)$ of
weak $\sigma$-twisted $V$-modules such that $T_g^k \circ U_g^k = 
id_{\mathcal{ C}_w^g(\vtk)}$ and $U_g^k \circ T_g^k = 
id_{\mathcal{ C}_w^\sigma(V)}$.  In particular, the categories 
$\mathcal{ C}_w^g(\vtk)$ and $\mathcal{ C}_w^\sigma(V)$ are isomorphic. Moreover,

(1) The restrictions of $T_g^k$ and $U_g^k$ to the category of
admissible $\sigma$-twisted $V$-modules $\mathcal{ C}_a^\sigma(V)$ and to the category of
admissible $g$-twisted $\vtk$-modules $\mathcal{ C}_a^g(\vtk)$,
respectively, give category isomorphisms. In particular, $\vtk$ is $g$-rational if and only if $V$ is $\sigma$-rational.

(2) The restrictions of $T_g^k$ and $U_g^k$ to the category of
ordinary $\sigma$-twisted $V$-modules $\mathcal{ C}^\sigma(V)$ and to the category of ordinary $g$-twisted $\vtk$-modules $\mathcal{ C}^g(\vtk)$, respectively, give
category isomorphisms.
\end{thm}

\begin{proof} It is trivial to verify $T_g^k \circ U_g^k = id_{\mathcal{
C}_w^g(\vtk)}$ and $U_g^k \circ T_g^k = id_{\mathcal{ C}_w^\sigma(V)}$ {}from the
definitions of the functors $T_g^k$ and $U_g^k$.  Parts (1) and (2) follow {}from Theorem \ref{main1}.  \end{proof}

Using the functor $T_g^k$ giving the isomorphism between the categories $\mathcal{C}^\sigma(V)$ and $\mathcal{C}^g(V^{\otimes k})$ as well as the actual construction of $g$-twisted $V^{\otimes k}$-modules from $\sigma$-twisted $V$-modules, we have a correspondence between graded traces of modules in $\mathcal{C}^\sigma (V)$ and modules in $\mathcal{C}^g(V^{\otimes k})$.  In particular, from (\ref{L(0)-conversion}), we have the following corollary. 

\begin{cor}\label{graded-dimension-corollary}  Let $g = (1 \; 2 \; \cdots \; k)$ for $k$ even.  Then $(M_\sigma, Y_\sigma)$ is an ordinary $\sigma$-twisted $V$-module with graded dimension 
\[ \mathrm{dim}_q M_\sigma =  tr_{M_\sigma}  q^{-c/24 + L^\sigma (0)} = q^{-c/24} \sum_{\lambda \in \mathbb{C}} \mathrm{dim} (M_\lambda) q^\lambda \]
if and only if $(T_g^k(M_\sigma), Y_g)$ is an ordinary $(1 \; 2 \; \cdots \; k)$-twisted $V^{\otimes k}$-module with graded dimension
\begin{eqnarray*}\mathrm{dim}_q T_g^k(M_\sigma) &=& tr_{T_g^k(M_\sigma)} q^{-kc/24 + L^g(0)} = tr_{T_g^k(M_\sigma)} q^{-kc/24 +\frac{1}{k} L^\sigma (0) + (k^2 - 1)c/24k} \\
&=&  \mathrm{dim}_{q^{1/k}} M_\sigma .
\end{eqnarray*}
\end{cor}

\end{document}